\documentclass[a4paper,11pt,reqno]{amsart}
\usepackage{upref}
\usepackage{amsthm}
\usepackage{amsfonts}
\usepackage{amssymb}
\usepackage{mathrsfs}
\usepackage{mathtools}
\usepackage{times}
\usepackage{cite}
\usepackage{bm}
\usepackage{graphicx}
\usepackage{multicol}
\usepackage{color}
\usepackage{pb-diagram}
\allowdisplaybreaks
 \newtheorem{theorem}{Theorem}

 \newtheorem{lemma}[theorem]{Lemma}
\theoremstyle{definition}
 \newtheorem{definition}[theorem]{Definition}
\theoremstyle{remark}
 
\begin{document}
\title[Double fibration transform]{Microlocal analysis of double fibration transforms with conjugate points}
\author[H.~Chihara]{Hiroyuki Chihara}
\address{College of Education, University of the Ryukyus, Nishihara, Okinawa 903-0213, Japan}
\email{hc@trevally.net}
\thanks{Supported by the JSPS Grant-in-Aid for Scientific Research \#23K03186.}
\subjclass[2020]{Primary 58J40, Secondary 53C65}
\keywords{double fibration transform, Fourier integral operator, Bolker condition, conjugate point}
\begin{abstract}
We study the structure of normal operators of double fibration transforms with conjugate points. Examples of double fibration transforms include  Radon transforms, $d$-plane transforms on the Euclidean space, geodesic X-ray transforms, light-ray transforms, and ray transforms defined by null bicharacteristics associated with real principal type operators. We show that, under certain stable conditions on the distribution of conjugate points, the normal operator splits into an elliptic pseudodifferential operator and several Fourier integral operators, depending on the degree of the conjugate points. These problems were first studied for geodesic X-ray transforms by Stefanov and Uhlmann (Analysis \& PDE, {\bf 5} (2012), pp.219--260). After that Holman and Uhlmann (Journal of Differential Geometry, {\bf 108} (2018), pp.459--494) proved refined results according to the degree of regular conjugate points.  
\end{abstract}
\maketitle
\section{Introduction}
\label{section:introduction}
We study the structure of the normal operators of double fibration transforms using microlocal analysis. In particular, when conjugate points exist, we split the normal operators into several parts according to the degree of conjugate points. We begin with the $d$-plane transforms on the Euclidean space to explain the background and our motivation.
\par
We now introduce the $d$-plane transform and state the inversion formula following Helgason's celebrated textbook \cite[Chapter~1, Section~6]{Helgason}. Let $\mathbb{R}^n$ with $n\geqq2$ be the Euclidean space, and let $d$ be a positive integer strictly less than $n$. Denote by $G_{d,n}$ and $G(d,n)$ the Grassmannian and the affine Grassmannian respectively, that is, $G_{d,n}$ is the set of all $d$-dimensional vector subspaces of $\mathbb{R}^n$ and $G(d,n)$ is the set of all $d$-dimensional affine subspaces in $\mathbb{R}^n$. For any $\sigma \in G_{d,n}$, we have the orthogonal direct sum 
$\mathbb{R}^n = \sigma\oplus\sigma^\perp$, 
where $\sigma^\perp$ is the orthogonal complement of $\sigma$ in $\mathbb{R}^n$. 
For any fixed $\sigma \in G_{d,n}$, 
we choose a coordinate system of $\mathbb{R}^n$ such as 
$x=x^\prime+x^{\prime\prime} \in \sigma\oplus\sigma^\perp$. 
The affine Grassmannian $G(d,n)$ is given by 
$$
G(d,n)
=
\{\sigma+x^{\prime\prime}: \sigma \in G_{d,n}, x^{\prime\prime} \in \sigma^\perp\}. 
$$    
The $d$-plane transform $R_df(\sigma+x^{\prime\prime})$ of an appropriate function $f(x)$ on $\mathbb{R}^n$ is defined by 
$$
R_df(\sigma+x^{\prime\prime})
:= 
\int_\sigma f(x^\prime+x^{\prime\prime}) dx^\prime,
$$
where $dx^\prime$ is the Lebesgue measure on $\sigma$. 
The formal adjoint of $R_d$ of a continuous function $\varphi$ on $G(d,n)$ 
is explicitly given by  
$$
R_d^\ast\varphi(x)
:=
\frac{1}{C(d,n)}
\int_{O(n)}
\varphi(x+k\cdot\sigma)
dk, 
$$
where $C(d,n)=(4\pi)^d\Gamma(n/2)/\Gamma((n-1)/2)$, 
$\Gamma(\cdot)$ is the gamma function, 
$O(n)$ is the orthogonal group, 
$dk$ is the normalized measure which is invariant under rotations, 
and $\sigma \in G_{d,n}$ is arbitrary. 
The inversion formula for $R_d$ is given as follows.
\begin{theorem}[{\cite[Chapter~1, Theorem~6.2]{Helgason}}]
\label{theorem:siggi}
For $f(x)=\mathcal{O}(\langle{x}\rangle^{-d-\varepsilon})$ on $\mathbb{R}^n$ for some $\varepsilon>0$, we have 
$$
f=(-\Delta_{\mathbb{R}^n})^{d/2}R_d^\ast{R_d}f,
$$
where $x=(x_1,\dotsc,x_n)\in\mathbb{R}^n$, 
$\langle{x}\rangle:=\sqrt{1+x_1^2+\dotsb+x_n^2}$, 
$-\Delta_{\mathbb{R}^n}:=-\partial_{x_1}^2-\dotsb-\partial_{x_n}^2$.  
\end{theorem}
We see that Theorem~\ref{theorem:siggi} says that the normal operator 
$R_d^\ast{R_d}$ for $R_d$ is an elliptic pseudodifferential operator $(-\Delta_{\mathbb{R}^n})^{-d/2}$ 
and that $R_d^\ast{R_d}$ has a local parametrix everywhere on $\mathbb{R}^n$. 
In other words, the inversion formula for $R_d$ holds 
since the normal operator for $R_d$ is an elliptic pseudodifferential operator. 
More precisely, 
the normal operator for $R_d$ is an elliptic pseudodifferential operator 
since the Euclidean space $\mathbb{R}^n$ has 
no conjugate points as a Riemannian manifold. 
Unfortunately, however, it might be impossible to notice this fact 
if one studies such a problem only on the Euclidean space. 
\par
We consider the normal operator of the geodesic X-ray transform on Riemannian manifolds. The motivation of the study of the normal operator comes from the study of the invertibility of the X-ray transform. We should refer three pioneering works \cite{PlamenGunther1,PlamenGunther2,PlamenGunther3} by Stefanov and Uhlmann. In \cite{PlamenGunther1} they studied the Geodesic X-ray transform for symmetric 2-tensors on simple Riemannian manifolds, proved that the normal operator is a system of pseudodifferential operators of order $-1$, obtained its principal symbol, its kernel and its parametrix, and established a stability estimate. 
In \cite{PlamenGunther2} they studied the Geodesic X-ray transform for symmetric 2-tensors on simple Riemannian manifolds in the real-analytic setting, and proved that the normal operator is s-injective, that is, the normal operator is injective modulo potential tensor fields. For this purpose, they proved that the normal operator is a system of analytic pseudodifferential operators of order $-1$, and construct its analytic parametrix. 
\par
In \cite{PlamenGunther3} they studied the Geodesic X-ray transform for scalar functions on Riemannian manifolds with fold caustics. More precisely, let $v_0$ be a fold conjugate vector at $p_0$, and let $q_0=\exp_{p_0}(t_0v_0)$ ($t_0\in\mathbb{R}$) be the conjugate point to $p_0$. They studied the normal operator in a small neighborbood of the geodesic segment joining $p_0$ and $q_0$, and proved that the normal operator $N$ admits the decomposition $N=A+F$, where $A$ is an elliptic pseudodifferential operator of order $-1$, $F$ is a Fourier integral operator of order $-n/2$, and $n$ is the dimension of the Riemannian manifold. The definition of fold caustics is as follows. 
A conjugate vector $v_0$ is of fold type if $D\exp_{p_0}(v_0)$ has corank one and
$\det\bigl(D\exp_{p_0}(v)\bigr)$ vanishes simply at $v=v_0$.  
Equivalently, the kernel of $D\exp_{p_0}(v_0)$ is one-dimensional and 
the second derivative of $\exp_{p_0}$ in the transverse direction is nondegenerate. 
In this case the image of the conjugate locus under $\exp_{p_0}$ forms a fold caustic. 
\par
Later Holman and Uhlmann studied the Geodesic X-ray transform for scalar functions on Riemannian manifolds with conjugate points in \cite{HolmanUhlmann}. The results of \cite{HolmanUhlmann} show that if all the conjugate points are regular, then the normal operator becomes a sum of an elliptic pseudodifferential operator of order $-1$ and and several Fourier integral operators according to the degree of conjugate points. See e.g., \cite{HolmanUhlmann,Warner} for the definition of regular and singular conjugate points. 
\par
Let $(M,g)$ be an $n$-dimensional compact non-trapping Riemannian manifold with strictly convex boundary. We say that a Riemannian manifold $(M,g)$ is nontrapping if all geodesics exit $M^\text{int}$ in finite time. Denote by $SM$ the unit tangent sphere bundle over $M$, and by $\nu(x)$ the unit outer normal vector at $x\in\partial{M}$. Set 
$$
\partial_-SM
:=
\{(x,v) \in SM : x\in\partial{M}, \langle{v,\nu(x)}\rangle<0\},
$$
which is the set of all the inward unit tangent vectors on $\partial{M}$.  
For any normal geodesic $\gamma$ there exists a unique element $(x,v) \in \partial_-SM$ 
such that $\bigl(\gamma(0),\dot{\gamma}(0)\bigr)=(x,v)$. 
The exit time of $\gamma$ is denoted by $\tau(x,v)$. 
We mean a geodesic with unit speed by a phrase ``normal geodesic''. 
$\partial_-SM$ can be identified with the space of all the normal geodesics on $M$. 
The geodesic X-ray transform $\mathcal{X}f(\gamma)=\mathcal{X}f(x,v)$ 
of $f \in \mathcal{D}(M^\text{int},\Omega_{M^\text{int}}^{1/2})$ is defined by 
$$
\mathcal{X}f(\gamma)=\mathcal{X}f(x,v)
:=
\left(
\int_0^{\tau(x,v)}
w\bigl(x,v,\gamma(t)\bigr)
\frac{f}{\lvert{dM}\rvert^{1/2}}\bigl(\gamma(t)\bigr)
dt
\right)
\lvert{d\partial_-SM(x,v)}\rvert^{1/2},
$$
where $w \in C^\infty(\partial_-SM \times M^\text{int})$ is a nowhere vanishing weight function, 
$dM$ is the Riemannian measure induced by the metric $g$, 
$\mathcal{D}(M^\text{int},\Omega_{M^\text{int}}^{1/2})$ 
is the space of all compactly supported smooth half densities on $M^\text{int}$, 
and $d\partial_-SM$ is the Riemannian measure induced by the Sasaki metric. 
We now state the conjugate points on $M$ 
in a seemingly different form as follows. 
Let $(x,v) \in \partial_-SM$, let $\gamma$ be the corresponding normal geodesic, 
and let  $y_0=\gamma(t_0)$ and $y_1=\gamma(t_1)$ with $t_0\ne{t_1}$. 
We say that $\bigl((x,v\bigr),y_0,y_1)$ is a conjugate triplet of degree $k=1,\dotsc,n-1$
if $y_1$ is a conjugate point of $y_0$ along $\gamma$ of degree $k$. 
Moreover we say that $\bigl((x,v\bigr),y_0,y_1)$ 
is a {\it regular} conjugate triplet of degree $k$ if there exist 
a neighborhood $V$ of $(x,v)$ in $\partial_-SM$, 
a neighborhood $U_0$ of $y_0$ in $M^\text{int}$, 
and 
a neighborhood $U_1$ of $y_0$ in $M^\text{int}$, 
such that every conjugate triplet 
$\bigl((x^\prime,v^\prime\bigr),y_0^\prime,y_1^\prime) \in V \times U_0 \times U_1$ 
is also of degree $k$. 
The main results of Holman and Uhlmann \cite{HolmanUhlmann} are the following. 
\begin{theorem}[{\cite[Theorem~4 in page 481]{HolmanUhlmann}}]
\label{theorem:hkust}
Suppose that all the conjugate triplets are regular. 
Then the normal operator $\mathcal{X}^\ast\mathcal{X}$ is decomposed as 
$$
\mathcal{X}^\ast\mathcal{X}
=
P
+
\sum_{k=1}^{n-1}
\sum_{l=1}^{M_k}
A_{k,l},
$$
where $P$ is a pseudodifferential operator of order $-1$, 
and for $k=1,\dots,n-1$ and $l=1,\dots,M_k$ 
$$
A_{k,l}
\in
\mathcal{I}^{-(n+1-k)/2}
(M^\text{int}{\times}M^\text{int},
C_{A_{k,l}}^\prime,
\Omega^{1/2}_{M^\text{int}{\times}M^\text{int}}), 
$$
$\mathcal{I}^{-(n+1-k)/2}(M^\text{int}{\times}M^\text{int},C_{A_{k,l}}^\prime,\Omega^{1/2}_{M^\text{int}{\times}M^\text{int}})$ 
is the standard notation of the class of Lagrangian distributions 
{\rm (}See, e.g., {\rm \cite{Hoermander4})},  
$C_{A_{k,l}}$ is the canonical relation of $A_{k,l}$ and does not contain a diagonal part of $T^\ast{M}{\times}T^\ast{M}$, 
$k (=1,\dotsc,n-1)$ describe the degree of the conjugate triplets, 
$l (=1,\dotsc,M_k)$ correspond to connected components of the set of regular conjugate triplets of degree $k$. 
\end{theorem} 
Theorem~\ref{theorem:hkust} shows that if there are conjugate points and all the conjugate triplets are regular, the normal operator $\mathcal{X}^\ast\mathcal{X}$ contains Fourier integral operators which are essentially different from elliptic pseudodifferential operators. The highest order term of the Fourier integral operator part is the term of $k=n-1$:
$$
\sum_{l=1}^{M_{n-1}}
A_{n-1,l}, 
\quad
A_{n-1,l}
\in
\mathcal{I}^{-1}
(M^\text{int}{\times}M^\text{int},
C_{A_{n-1,l}}^\prime,
\Omega^{1/2}_{M^\text{int}{\times}M^\text{int}}).
$$
The order of this part is $-1$ and the same as that of $\mathcal{X}^\ast\mathcal{X}$. So we cannot expect the existence of a parametrix of the normal operator. What happens to the normal operators if conjugate points arise? Actually so-called the Bolker condition fails to hold and several Fourier integral operators arise in the normal operator when there are conjugate points. 
\par
In this paper, we show Theorem~\ref{theorem:hkust} in the context of double fibration transforms, which are generalizations of integral transforms arising in integral geometry and the microlocal analysis of various real principal type pseudodifferential operators. The double fibration approach to integral geometry was initiated by Helgason \cite{Helgason}. See also Gelfand, Graev and Shapiro \cite{GelfandGraevShapiro}, Guillemin and Sternberg \cite{GuilleminSternberg}, Guillemin \cite{Guillemin}, Quinto \cite{Quinto}, and some other references. We now state the setting of the double fibration in this paper following the recent interesting and important paper \cite{MazzucchelliSaloTzou} by Mazzucchelli, Salo and Tzou. 
\par
Let $\mathcal{G}$ and $X$ be oriented smooth manifolds without boundaries. 
Set $N:=\dim(\mathcal{G})$ and $n:=\dim(X)$ for short. 
Denote by $d\mathcal{G}$ and $dX$ the orientation form of $\mathcal{G}$ 
and the orientation form of $X$ respectively. 
Let $Z$ be an oriented embedded submanifold of $\mathcal{G}{\times}X$, 
and denote the orientation form on $Z$ by $dZ$.    
Assume that $N+n > \dim(Z) > N \geqq n \geqq 2$, and 
set $n^\prime:=\dim(Z)-N$ and $n^{\prime\prime}:=n-n^\prime$. 
Then $\dim(Z)=N+n^\prime$, $n=n^\prime+n^{\prime\prime}$ and $n^\prime,n^{\prime\prime}=1,\dotsc,n-1$. 
$$
\begin{diagram}
\node[2]{Z} \arrow{sw,t}{\pi_\mathcal{G}} \arrow{se,t}{\pi_X}
\\
\node{\mathcal{G}} \node[2]{X}  
\end{diagram}
$$
\begin{center}
Fig.1\ Double fibration. 
\end{center}
We assume that $Z$ is a double fibration, that is, 
the natural projections
$\pi_\mathcal{G}:Z{\rightarrow}\mathcal{G}$ and 
$\pi_X:Z{\rightarrow}X$ are submersions respectively. 
Then $G_z:=\pi_X\circ\pi_\mathcal{G}^{-1}(z)$ becomes an 
$n^\prime$-dimensional embedded submanifold of X for any $z \in \mathcal{G}$, 
and $H_x:=\pi_\mathcal{G}\circ\pi_X^{-1}(x)$ forms an 
$(N-n^{\prime\prime})$-dimensional embedded submanifold of $\mathcal{G}$ for any $x \in X$. 
\par
We now state the induced orientation forms on $G_z$ and $H_x$. 
Fix arbitrary $(z,x) \in Z$, 
and let $\{v_1,\dotsc,v_n\}$ and $\{w_1,\dotsc,w_N\}$ be bases of 
$T_xX$ and $T_z\mathcal{G}$ respectively such that 
$$
T_{(z,x)}Z
=
\operatorname{span}\langle{v_1,\dotsc,v_{n^\prime},w_1,\dotsc,w_N}\rangle
=
\operatorname{span}\langle{v_1,\dotsc,v_n,w_1,\dotsc,w_{N-n^{\prime\prime}}}\rangle.
$$
The induced orientation forms $dG_z$ on $G_z$ and $dH_x$ on $H_x$ are given by 
\begin{align*}
  dG_z\bigl(d\pi_X(v_1),\dotsc,d\pi_X(v_{n^\prime})\bigr) :
& =
  dZ_{\pi_\mathcal{G}^{-1}(z)}(v_1,\dotsc,v_{n^\prime})
\\
& =
  \frac{dZ(v_1,\dotsc,v_{n^\prime},w_1,\dotsc,w_N)}{d\mathcal{G}\bigl(d\pi_\mathcal{G}(w_1),\dotsc,d\pi_\mathcal{G}(w_N)\bigr)},
\\ 
  dH_x\bigl(d\pi_\mathcal{G}(w_1),\dotsc,d\pi_\mathcal{G}(w_{N-n^{\prime\prime}})\bigr) :
& =
  dZ_{\pi_X^{-1}(x)}(w_1,\dotsc,w_{N-n^{\prime\prime}})
\\
& =
  \frac{dZ(v_1,\dotsc,v_n,w_1,\dotsc,w_{N-n^{\prime\prime}})}{dX\bigl(d\pi_X(v_1),\dotsc,d\pi_X(v_n)\bigr)}.
\end{align*}
\par
We state the definition of double fibration transform 
associated with the double fibration $Z$, 
its adjoint, and their mapping properties. 
We denote by $\mathscr{D}(X,\Omega_X^{1/2})$ 
the space of compactly supported smooth half densities 
on X equipped with the standard inductive limit topology, 
and by $\mathscr{E}(X,\Omega_X^{1/2})$ the space of smooth half densities 
on X equipped with the standard Fr\'echet space topology. 
Their topological duals are denoted by $\mathscr{D}^\prime(X,\Omega_X^{1/2})$ 
and by $\mathscr{E}^\prime(X,\Omega_X^{1/2})$ respectively. 
Further 
$\mathscr{D}(\mathcal{G},\Omega_\mathcal{G}^{1/2})$, 
$\mathscr{D}^\prime(\mathcal{G},\Omega_\mathcal{G}^{1/2})$, 
$\mathscr{E}(\mathcal{G},\Omega_\mathcal{G}^{1/2})$, 
and 
$\mathscr{E}^\prime(\mathcal{G},\Omega_\mathcal{G}^{1/2})$ 
are defined similarly. 
Suppose that a weight function $\kappa(z,x) \in C^\infty(\mathcal{G}{\times}X)$ is nowhere vanishing. 
A double fibration transform $\mathcal{R}$ associated with the double fibration $Z$ is defined by 
$$
\mathcal{R}f(z)
:=
\left(
\int_{G_z}
\kappa(z,x)
\frac{f}{\lvert{dX}\rvert^{1/2}}(x)
dG_z(x)
\right)
\lvert{d\mathcal{G}(z)}\rvert^{1/2}
$$
for $f \in \mathscr{D}(X,\Omega_X^{1/2})$. 
The adjoint $\mathcal{R}^\ast$ is given by 
$$
\mathcal{R}^\ast{u(x)}
=
\left(
\int_{H_x}
\overline{\kappa(z,x)}
\frac{u}{\lvert{d\mathcal{G}}\rvert^{1/2}}(z)
dH_x(z)
\right)
\lvert{dX(x)}\rvert^{1/2}
$$
for $u \in \mathscr{D}(\mathcal{G},\Omega_\mathcal{G}^{1/2})$. 
Then we deduce that 
$$
\mathcal{R}:\mathscr{D}(X,\Omega_X^{1/2}) \rightarrow \mathscr{E}(\mathcal{G},\Omega_\mathcal{G}^{1/2}), 
\quad
\mathcal{R}^\ast:\mathscr{D}(\mathcal{G},\Omega_\mathcal{G}^{1/2}) \rightarrow \mathscr{E}(X,\Omega_X^{1/2}), 
$$
are continuous linear mappings, and so are  
$$
\mathcal{R}:\mathscr{E}^\prime(X,\Omega_X^{1/2}) \rightarrow \mathscr{D}^\prime(\mathcal{G},\Omega_\mathcal{G}^{1/2}), 
\quad
\mathcal{R}^\ast:\mathscr{E}^\prime(\mathcal{G},\Omega_\mathcal{G}^{1/2}) \rightarrow \mathscr{D}^\prime(X,\Omega_X^{1/2}).
$$
More precisely $\mathcal{R}$ and $\mathcal{R}^\ast$ are elliptic Fourier integral operators. 
Let $N^\ast{Z}$ be the conormal bundle of $Z$. 
More precisely we set 
$N^\ast_{(z,x)}{Z}:=T^\ast_{(z,x)}(\mathcal{G}{\times}X)/T^\ast_{(z,x)}{Z}$ 
for any $(z,x) \in Z$, 
and regard it as if it were a vector subspace of $T^\ast_{(z,x)}(\mathcal{G}{\times}X)$. 
Set 
\begin{align*}
  (N^\ast{Z}\setminus0)^T
& :=
  \{(x,\eta,z,\zeta) : (z,\zeta,x,\eta) \in N^\ast{Z}\setminus0\}, 
\\
  (N^\ast{Z}\setminus0)^\prime
& :=
  \{(z,\zeta,x,-\eta) : (z,\zeta,x,\eta) \in N^\ast{Z}\setminus0\}. 
\end{align*}
Combining the results of \cite{MazzucchelliSaloTzou} and \cite{Hoermander4} we have the following. 
\begin{theorem}[{\cite[Theorem~2.2]{MazzucchelliSaloTzou} and \cite[Theorem~25.2.2]{Hoermander4}}]
\label{theorem:fios}
Suppose that $Z$ is a double fibration with $\dim(Z)=N+n^\prime$. 
Then $\mathcal{R}$ and $\mathcal{R}^\ast$ are elliptic Fourier integral operators of order $-(N+2n^\prime-n)/4$ with canonical relations $(N^\ast{Z}\setminus0)^\prime$ 
and $\bigl((N^\ast{Z}\setminus0)^T\bigr)^\prime$ respectively. 
More precisely  
\begin{align*}
  \mathcal{R}
& \in 
  \mathcal{I}^{-(N+2n^\prime-n)/4}
  \bigl(\mathcal{G}{\times}X,N^\ast{Z}\setminus0;\Omega_{\mathcal{G}{\times}X}^{1/2}\bigr),
\\
  \mathcal{R}^\ast
& \in 
  \mathcal{I}^{-(N+2n^\prime-n)/4}
  \bigl(X\times\mathcal{G},(N^\ast{Z}\setminus0)^T;\Omega_{X\times\mathcal{G}}^{1/2}\bigr), 
\end{align*}
where 
\begin{align*}
  N^\ast{Z}\setminus0 
& =
  \bigl\{
  \bigl(z,A(z,x)\eta,x,\eta\bigr) : 
  (z,x) \in Z, \eta \in N^\ast_xG_z\setminus\{0\}
  \bigr\}
\\
& =
  \bigl\{
  \bigl(z,\zeta,x,B(z,x)\zeta\bigr) : 
  (z,x) \in Z, \zeta \in N^\ast_zH_x\setminus\{0\}
  \bigr\},
\end{align*}
$A(z,x) \in \operatorname{Hom}(N^\ast_xG_z,T^\ast_z\mathcal{G})$ 
and 
$B(z,x) \in \operatorname{Hom}(N^\ast_zH_x,T^\ast_xX)$ 
smoothly depend on $(z,x) \in Z$ respectively. 
\end{theorem}
The Schwartz kernels of $\mathcal{R}$ and $\mathcal{R}^\ast$ are conormal distributions. 
See e.g., \cite[Section~18.2]{Hoermander3} for conormal distributions. 
We give the concrete local expressions of $A(z,x)$ and $B(z,x)$ in the next section. 
$$
\begin{diagram}
\node[2]{(N^\ast{Z}\setminus0)^\prime} \arrow{sw,t}{\pi_\text{L}} \arrow{se,t}{\pi_\text{R}}
\\
\node{T^\ast\mathcal{G}\setminus0} \node[2]{T^\ast{X}\setminus0}  
\end{diagram}
$$
\begin{center}
Fig.2\ Conormal bundle of double fibration and natural projections. 
\end{center}
\par
There are numerous examples of double fibration transforms 
as is introduced in \cite{MazzucchelliSaloTzou}. 
See also 
\cite{Siggi,Chihara1} for the $d$-plane transforms on $\mathbb{R}^n$,  
\cite{Chihara2,HolmanUhlmann,StefanovUhlmannVasy} for the geodesic X-ray transform, 
\cite{FeizmohammadiIlmavirtaOksanen,LassasOksanenStefanovUhlmann,VasyWang} 
for the light ray transform, 
\cite{OksanenSaloStefanovUhlmann} for the analysis of real principal type operators, and references therein. 
In this paper we make use of microlocal analysis 
based on conormal distributions and more generally Lagrangian distributions. 
See H\"ormander's textbooks \cite{Hoermander3,Hoermander4} for this. 
The explanations in \cite{KrishnanQuinto} and \cite[Section~2]{WebberHolmanQuinto} provide an excellent introduction to microlocal analysis as it relates to integral geometry and tomography. 
\par
The plan of this paper is as follows. In Section~2 we introduce local expressions of the double fibration $Z$ and $Z$-conjugate points following \cite[Section~3]{MazzucchelliSaloTzou}, and confirm basic facts including $A(z,x)$ and $B(z,x)$. In Section~3 we show that the normal operator $\mathcal{R}^\ast\mathcal{R}$ becomes an elliptic pseudodifferential operator of order $-n^\prime$ provided that there are no $Z$-conjugate points (Theorem~\ref{theorem:normal1}). Finally, in Section~4 we show that if all the conjugate points are regular (See Definition~\ref{theorem:z-conjugate}) and some additional condition on the sets of regular conjugate points holds, then we obtain the decomposition of $\mathcal{R}^\ast\mathcal{R}$ like Theorem~\ref{theorem:hkust} (see Theorem~\ref{theorem:normal2}).  
\section{Preliminaries}
\label{section:local}
In this section we introduce local expressions of $Z$ and $Z$-conjugate points following 
\cite{MazzucchelliSaloTzou}, and confirm basic facts used later. 
Since $\dim(Z)=N+n^\prime=(N+n)-n^{\prime\prime}$, 
there exist a local coordinate system  
$(z,x)=(z^\prime,z^{\prime\prime},x^\prime,x^{\prime\prime})\in\mathbb{R}^{N-n^{\prime\prime}}\times\mathbb{R}^{n^{\prime\prime}}\times\mathbb{R}^{n^\prime}\times\mathbb{R}^{n^{\prime\prime}}$, 
and $\mathbb{R}^{n^{\prime\prime}}$-valued smooth functions 
$\phi(z,x^\prime)$ and $b(x,z^\prime)$ such that $Z$ has local expressions  
$\{x^{\prime\prime}=\phi(z,x^\prime)\}$ and $\{z^{\prime\prime}=b(x,z^\prime)\}$. 
We have local expressions of linear mappings 
$A(z,x) \in \operatorname{Hom}(N^\ast_xG_z,T^\ast_z\mathcal{G})$ 
and 
$B(z,x) \in \operatorname{Hom}(N^\ast_zH_z,T^\ast_xX)$, 
and these expressions show that both $A(z,x)$ and $B(z,x)$ depend on $(z,x) \in Z$ smoothly. 
\begin{lemma}[{\cite[Lemmas~2.4, 2.5 and 2.6]{MazzucchelliSaloTzou}}]
\label{theorem:localexpression}
Suppose that $Z$ is a double fibration. 
\begin{itemize}
\item 
If $Z$ is locally given by $\{x^{\prime\prime}=\phi(z,x^\prime)\}$, then in these coordinates 
\begin{align*}
  T_xG_z
& =
  \bigl\{
  \bigl(v^\prime,\phi_{x^\prime}(z,x^\prime)v^\prime\bigr) 
  : 
  v^\prime \in \mathbb{R}^{n^\prime}
  \bigr\},
\\
  N^\ast_xG_z
& =
  \bigl\{
  \bigl(-\phi_{x^\prime}(z,x^\prime)^T\eta^{\prime\prime},\eta^{\prime\prime}\bigr)
  :
  \eta^{\prime\prime} \in \mathbb{R}^{n^{\prime\prime}}
  \bigr\},
\\
  N^\ast_{(z,x)}Z
& =
  \Bigl\{
  \Bigl(
  -\phi_z(z,x^\prime)^T\eta^{\prime\prime},
  \bigl(-\phi_{x^\prime}(z,x^\prime)^T\eta^{\prime\prime},\eta^{\prime\prime}\bigr)
  \Bigr)
  : 
  \eta^{\prime\prime} \in \mathbb{R}^{n^{\prime\prime}}  
  \Bigr\},     
\end{align*}
$$
A(z,x)
\begin{bmatrix}
-\phi_{x^\prime}(z,x^\prime)^T
\\
I_{n^{\prime\prime}} 
\end{bmatrix}
\eta^{\prime\prime}
=
-\phi_z(z,x^\prime)^T\eta^{\prime\prime},
\quad
\eta^{\prime\prime}\in\mathbb{R}^{n^{\prime\prime}},
$$
$\operatorname{Ran}\bigl(A(z,x)\bigr)=N^\ast_zH_x$, 
and $\operatorname{rank}\bigl(A(z,x)\bigr)=\operatorname{rank}\bigl(\phi_z(z,x^\prime)\bigr)=n^{\prime\prime}$, where $I_{n^{\prime\prime}}$ is the $n^{\prime\prime}{\times}n^{\prime\prime}$ identity matrix. 
\item 
If $Z$ is locally given by $\{z^{\prime\prime}=b(x,x^\prime)\}$, then in these coordinates 
\begin{align*}
  T_zH_x
& =
  \bigl\{
  \bigl(w^\prime,b_{z^\prime}(x,z^\prime)v^\prime\bigr) 
  : 
  w^\prime \in \mathbb{R}^{N-n^{\prime\prime}}
  \bigr\},
\\
  N^\ast_zH_x
& =
  \bigl\{
  \bigl(-b_{z^\prime}(x,z^\prime)^T\zeta^{\prime\prime},\zeta^{\prime\prime}\bigr)
  :
  \zeta^{\prime\prime} \in \mathbb{R}^{n^{\prime\prime}}
  \bigr\},
\\
  N^\ast_{(z,x)}Z
& =
  \Bigl\{
  \Bigl(
  \bigl(-b_{z^\prime}(x,z^\prime)^T\zeta^{\prime\prime},\zeta^{\prime\prime}\bigr),
   -b_x(x,z^\prime)^T\zeta^{\prime\prime}
  \Bigr)
  : 
  \zeta^{\prime\prime} \in \mathbb{R}^{n^{\prime\prime}}  
  \Bigr\},     
\end{align*}
$$
B(z,x)
\begin{bmatrix}
-b_{z^\prime}(x,z^\prime)^T
\\
I_{n^{\prime\prime}} 
\end{bmatrix}
\zeta^{\prime\prime}
=
-b_x(x,z^\prime)^T\zeta^{\prime\prime},
\quad
\zeta^{\prime\prime}\in\mathbb{R}^{n^{\prime\prime}},
$$
$\operatorname{Ran}\bigl(B(z,x)\bigr)=N^\ast_xG_z$, 
and $\operatorname{rank}\bigl(B(z,x)\bigr)=\operatorname{rank}\bigl(b_x(x,z^\prime)\bigr)=n^{\prime\prime}$.
\end{itemize}
\end{lemma}
We will introduce the notion of conjugate points associated with the double fibration $Z$. 
For this purpose we will make use of 
$$
A(z,x)^\ast 
\in 
\operatorname{Hom}\bigl(T_z\mathcal{G},(N^\ast_xG_z)^\ast\bigr)
\simeq
\operatorname{Hom}\bigl(T_z\mathcal{G},N_xG_z\bigr), 
\quad
(z,x) \in Z, 
$$ 
where $N_xG_z:=T_xX/T_xG_z$ is the normal space of $G_z$ at $x \in G_z$, and $\simeq$ means that the both hand sides are identified since $N_xG_z$ and $(N^\ast_xG_z)^\ast=\bigl((N_xG_z)^\ast\bigr)^\ast$ are naturally identified. 
The following facts seem to be obvious, 
but we will check this is true just in case. 
\begin{lemma}
\label{theorem:Aast}
Suppose that $Z$ is a double fibration. Then we have 
$$
T_zH_x
\simeq
\operatorname{Ker}\bigl(A(z,x)^\ast\bigr),
\quad
T_xG_z
\simeq
\operatorname{Ker}\bigl(B(z,x)^\ast\bigr),
$$
for any $(z,x) \in Z$. 
\end{lemma}
\begin{proof}
We shall show only $T_zH_x\simeq\operatorname{Ker}\bigl(A(z,x)^\ast\bigr)$. 
We can prove $T_xG_z\simeq\operatorname{Ker}\bigl(B(z,x)^\ast\bigr)$ similarly. 
Fix arbitrary $(z,x) \in Z$ and pick up arbitrary $w \in T_zH_x \subset T_z\mathcal{G}$. 
We first show that $w \in \operatorname{Ker}\bigl(A(z,x)^\ast\bigr)$. 
We consider a short curve in $H_x$ of the form 
$$
z(s)=z+sw+\mathcal{O}(s^2)
\quad\text{near}\quad
s=0.
$$
We may assume that $Z$ is locally given by $\{x^{\prime\prime}=\phi(z,x^\prime)\}$. 
The curve $z(s)$ satisfies $x^{\prime\prime}=\phi\bigl(z(s),x^\prime\bigr)$ 
since $\bigl(z(s),x\bigr) \in Z$ near $s=0$. Hence we deduce that 
\begin{equation}
0
=
\frac{d}{ds}\bigg\vert_{s=0}
\phi\bigl(z(s),x^\prime\bigr)
=
\frac{d}{ds}\bigg\vert_{s=0}
\phi\bigl(z+sw+\mathcal{O}(s^2),x^\prime\bigr)
=
\phi_z(z,x^\prime)w.
\label{equation:zero1} 
\end{equation}
Lemma~\ref{theorem:localexpression} implies that for any $\eta \in N^\ast_xG_z$ 
there exists $\eta^{\prime\prime}$ in appropriate coordinates such that 
$A(z,x)\eta=-\phi_z(z,x^\prime)^T\eta^{\prime\prime}$. 
Then for any $\eta \in N^\ast_xG_z$ and for any $\tilde{w} \in T_z\mathcal{G}$, 
we have 
$$
\eta\bigl(A(z,x)^\ast\tilde{w}\bigr)
=
A(z,x)\eta(\tilde{w})
=
-\phi_z(z,x^\prime)^T\eta^{\prime\prime}(\tilde{w})
=
-\eta^{\prime\prime}\bigl(\phi_z(z,x^\prime)\tilde{w}\bigr). 
$$
Substitute \eqref{equation:zero1} into the above. 
Then we deduce that $A(z,x)^\ast{w}=0$ in $N_xG_z$. 
Therefore we obtain $w \in \operatorname{Ker}\bigl(A(z,x)^\ast\bigr)$ and 
$T_zH_x \subset \operatorname{Ker}\bigl(A(z,x)^\ast\bigr)$. 
\par
Conversely, we assume that $w \in \operatorname{Ker}\bigl(A(z,x)^\ast\bigr) \subset T_z\mathcal{G}$. 
Then we immediately deduce that $w \in T_zH_x$ since 
$\zeta(w)=0$ for any $\zeta \in \operatorname{Ran}\bigl(A(z,x)\bigr)=N^\ast_zH_x=T^\ast_z\mathcal{G}/T^\ast_z{H_x}$. 
Hence we have $\operatorname{Ker}\bigl(A(z,x)^\ast\bigr) \subset T_zH_x$. 
This completes the proof. 
\end{proof}
We now introduce the notion of a $Z$-conjugate triplet. 
Suppose that $Z$ is a double fibration. 
Fix arbitrary $(z,w) \in T\mathcal{G}$, and consider a curve in $\mathcal{G}$ of the form 
$$
z(s)=z+sw+\mathcal{O}(s^2)
\quad\text{near}\quad
s=0. 
$$
Then $(G_{z(s)})$ is said to be a variation of $G_z$, 
and the variation field $J_w : G_z \rightarrow (N^\ast_xG_z)^\ast$ 
associated to $(G_{z(s)})$ is defined by 
$$
J_w(x)
:=
A(z,x)^\ast{w}
\in 
(N^\ast_xG_z)^\ast
\simeq
N_xG_z
=
T_xX/T_xG_z
$$
for $x \in G_z$. 
For $z \in \mathcal{G}$ and $x,y \in G_z$, set 
$$
V_z(x,y)
:=
\{J_w(x) : w \in T_z\mathcal{G}, J_w(y)=0\}. 
$$
Note that $\dim\bigl(V_z(x,y)\bigr) \leqq n^{\prime\prime}$ holds 
since $\operatorname{rank}\bigl(A(z,x)^\ast\bigr)=n^{\prime\prime}$. 
We have 
$$
V_z(x,x)=\{J_w(x) : w \in T_z\mathcal{G}, J_w(x)=0\}=\{0\}.
$$
Further we deduce that 
$$
\dim\bigl(V_z(x,y)\bigr) = \dim\bigl(V_z(y,x)\bigr)
$$
for any $z \in \mathcal{G}$ and $x,y \in G_z$ since 
$$
\operatorname{rank}\bigl(A(z,x)^\ast\bigr)
=
\operatorname{rank}\bigl(A(z,y)^\ast\bigr)
=
n^{\prime\prime}, 
$$
$$
\dim\Bigl(\operatorname{Ker}\bigl(A(z,x)^\ast\bigr)\Bigr)
=
\dim\Bigl(\operatorname{Ker}\bigl(A(z,y)^\ast\bigr)\Bigr)
=
N-n^{\prime\prime}. 
$$
We now state the definition of a $Z$-conjugate triplet with reference to 
\cite{HolmanUhlmann,MazzucchelliSaloTzou}.  
\begin{definition}
\label{theorem:z-conjugate} 
Suppose that $Z$ is a double fibration and $N \geqq 2n^{\prime\prime}$. 
Let $k=1,\dotsc,n^{\prime\prime}$. 
\begin{itemize}
\item 
Let $z \in \mathcal{G}$ and let $x,y \in G_z$ with $x{\ne}y$.  
We say that $(z;x,y)$ is a $Z$-conjugate triplet of degree $k$ 
if $\dim\bigl(V_z(x,y)\bigr)=n^{\prime\prime}-k$. 
\item
We say that a $Z$-conjugate triplet $(z;x,y)$ of degree $k$ is regular if there exist 
a neighborhood $U_x$ of $x$ in $X$, 
a neighborhood $U_y$ of $y$ in $X$, 
and 
a neighborhood $W_z$ of $z$ in $\mathcal{G}$ 
such that any $Z$-conjugate triplet 
$(\tilde{z};\tilde{x},\tilde{y}) \in W_z{\times}U_x{\times}U_y$ 
is also of degree $k$. 
The set of all the regular $Z$-conjugate triplets of degree $k$ is denoted by $C_{R,k}$. 
Set $C_R:=\cup_{k=1}^{n^{\prime\prime}}C_{R,k}$.
\item 
The set of all the $Z$-conjugate triplets which are not regular is denoted by $C_S$. 
\end{itemize}
\end{definition}
The fact that the given $(z;x,y)$ is not a $Z$-conjugate triplet is characterized in 
\cite{MazzucchelliSaloTzou}. 
\begin{lemma}[{\cite[Lemma~3.1]{MazzucchelliSaloTzou}}]
\label{theorem:notzconjugate}
Suppose that $Z$ is a double fibration and $N \geqq 2n^{\prime\prime}$. 
Let $z \in \mathcal{G}$, and let $x,y \in G_z$ such that $x{\ne}y$. 
The following conditions are mutually equivalent.
\begin{itemize}
\item[{\rm (a)}] 
$(z;x,y)$ is not a $Z$-conjugate triplet. 
\item[{\rm (b)}] 
$\dim\bigl(V_z(x,y)\bigr)=n^{\prime\prime}$. 
\item[{\rm (c)}] 
$\dim
\bigl(
\operatorname{Ker}(A(z,x)^\ast)
\cap
\operatorname{Ker}(A(z,y)^\ast)
\bigr)=N-2n^{\prime\prime}$, i.e., $\dim(T_zH_x{\cap}T_zH_y)=N-2n^{\prime\prime}$. 
\item[{\rm (d)}] 
$\operatorname{Ker}(A(z,x)^\ast)+\operatorname{Ker}(A(z,y)^\ast)=T_z\mathcal{G}$, 
i.e., $T_zH_x+T_zH_y=T_z\mathcal{G}$. 
\end{itemize}
\end{lemma}
Similarly we have the basic properties of $Z$-conjugate triplets as follows. 
\begin{lemma}
\label{theorem:zconjugateproperty}
Suppose that $Z$ is a double fibration and $N \geqq 2n^{\prime\prime}$. 
Let $z \in Z$, and let $x,y \in G_z$ with $x{\ne}y$. 
Pick up a basis of $T_z\mathcal{G}$ including a basis of 
$\operatorname{Ker}\bigl(A(z,x)^\ast\bigr)\cup\operatorname{Ker}\bigl(A(z,y)^\ast\bigr)$. 
Using this basis, we can see 
$T_z\mathcal{G}/\operatorname{Ker}\bigl(A(z,x)^\ast\bigr)$ and 
$T_z\mathcal{G}/\operatorname{Ker}\bigl(A(z,x)^\ast\bigr)$ as vector subspaces of $T_z\mathcal{G}$. 
Fix arbitrary $k=1,\dotsc,n^{\prime\prime}$. 
Then the following conditions are mutually equivalent. 
\begin{itemize}
\item[{\rm (a)}] 
$(z;x,y)$ is a conjugate triplet of degree $k$, i.e., 
$\dim\bigl(V_z(x,y)\bigr)=n^{\prime\prime}-k$.  
\item[{\rm (b)}] 
$\dim
\Bigl(
\bigl(
T_z\mathcal{G}/\operatorname{Ker}(A(z,x)^\ast)
\bigr)
\cap
\bigl(
T_z\mathcal{G}/\operatorname{Ker}(A(z,y)^\ast)
\bigr)
\Bigr)=k$.
\item[{\rm (c)}] 
$\dim
\bigl(
\operatorname{Ker}(A(z,x)^\ast)
\cap
\operatorname{Ker}(A(z,y)^\ast)
\bigr)=N-2n^{\prime\prime}+k$.
\item[{\rm (d)}] 
$\dim(N^\ast_zH_x{\cap}N^\ast_zH_y)=k$. 
\end{itemize}
\end{lemma}
\begin{proof}
If we set 
$$
W_z(x,y):=\bigl(T_z\mathcal{G}/\operatorname{Ker}(A(z,x)^\ast)\bigr)\cap\operatorname{Ker}(A(z,y)^\ast)
$$ 
for short, then we have $V_z(x,y)=A(z,x)^\ast\bigl(W_z(x,y)\bigr)$, 
$A(z,x)^\ast: W_z(x,y) \rightarrow V_z(x,y)$ is injective, and 
$\dim\bigl(V_z(x,y)\bigr)=\dim\bigl(W_z(x,y)\bigr)$.  
We first recall some elementary facts:
\begin{equation}
\dim\bigl(\operatorname{Ker}(A(z,x)^\ast)\bigr)=\dim(T_z\mathcal{G})-\operatorname{rank}(A(z,x)^\ast)=N-n^{\prime\prime},
\label{equation:fact81}
\end{equation}
\begin{equation}
\dim\bigl(T_z\mathcal{G}/\operatorname{Ker}(A(z,x)^\ast)\bigr)=n^{\prime\prime}, 
\label{equation:fact82}
\end{equation}
\begin{equation}
T_z\mathcal{G}/\operatorname{Ker}(A(z,x)^\ast)
=
W_z(x,y)
\oplus
\bigl(
T_z\mathcal{G}/\operatorname{Ker}(A(z,x)^\ast)
\cap
(T_z\mathcal{G}/\operatorname{Ker}(A(z,y)^\ast)
\bigr),
\label{equation:fact83}
\end{equation}
\begin{equation}
\operatorname{Ker}(A(z,y)^\ast)
=
W_z(x,y)
\oplus
\bigl(\operatorname{Ker}(A(z,x)^\ast)\cap\operatorname{Ker}(A(z,y)^\ast)\bigr).
\label{equation:fact84}
\end{equation}
\eqref{equation:fact82} and  \eqref{equation:fact83} prove the equivalence of (a) and (b), 
and that \eqref{equation:fact81} and  \eqref{equation:fact84} show the equivalence of (a) and (c). 
It is easy to see the equivalence of (b) and (d). 
\end{proof}
We now state a few remarks. 
Recall that 
$$
A(z,x)^\ast \in \operatorname{Hom}(T_z\mathcal{G},N_xG_x), 
\quad
N_zH_z=T_z\mathcal{G}/T_zH_x \simeq T_z\mathcal{G}/\operatorname{Ker}\bigl(A(z,x)^\ast\bigr).
$$ 
If we restrict $A(z,x)^\ast$ on $N_zH_x$, 
we can see this restricted mapping as a vector space isomorphism from $N_zH_x$ onto $N_xG_z$, 
and we can define its inverse denoted by $\bigl(A(z,x)^\ast\vert_{N_zH_x}\bigr)^{-1}$. 
So we can define a composition 
$A(z,y)^\ast\circ\bigl(A(z,x)^\ast\vert_{N_zH_x}\bigr)^{-1} \in \operatorname{Hom}(N_xG_z,N_yG_z)$ 
provided that $z \in H_x{\cap}H_y$. 
If $(z;x,y)$ is a conjugate triplet of degree $k$, 
(c) and (b) of Lemma~\ref{theorem:zconjugateproperty} shows that 
\begin{equation}
\dim(H_x{\cap}H_y)=N-2n^{\prime\prime}+k, 
\quad
\operatorname{rank}
\Bigl(A(z,y)^\ast\circ\bigl(A(z,x)^\ast\vert_{N_zH_x}\bigr)^{-1}\Bigr)
=
k
\label{equation:rankisk}
\end{equation}
respectively.
\par
We now consider the characterization of $Z$-conjugate triplets 
in terms of local expressions of $Z$. 
Suppose that $(z_0;x_0,y_0) \in \mathcal{G}{\times}X{\times}X$ is a 
$Z$-conjugate triplet of degree $k=1,\dotsc,n^{\prime\prime}$. 
Then we have $x_0,y_0 \in G_{z_0}$, $z_0 \in H_{x_0}{\cap}H_{y_0}$, 
and $\operatorname{dim}\bigl(V_{z_0}(x_0,y_0)\bigr)=n^{\prime\prime}-k$, 
and there exist $\mathbb{R}^{n^{\prime\prime}}$-valued functions $\phi(z,x^\prime)$ 
and $\psi(z,y^\prime)$ such that $Z$ is expressed as 
$\{x^{\prime\prime}=\phi(z,x^\prime)\}$ near $(z_0,x_0)$ 
and $\{y^{\prime\prime}=\psi(z,y^\prime)\}$ near $(z_0,y_0)$. 
In terms of $\phi$ and $\psi$, we have 
$$
V_{z_0}(x_0,y_0)
=
\bigl\{\phi_z(z_0,x_0^\prime)w : w \in \operatorname{Ker}\bigl(\psi_z(z_0,y_0^\prime)\bigr)\bigr\}. 
$$
Denote the entries of $\phi$ and $\psi$ 
by $\phi^{(i)}$ and $\psi^{(i)}$ with $i=1,\dotsc,n^{\prime\prime}$ respectively. 
We see $\phi_z$ and $\psi_z$ as $n^{\prime\prime}{\times}N$ matrices. 
Then $\{\phi_z^{(1)}(z,x^\prime),\dotsc,\phi_z^{(n^{\prime\prime})}(z,x^\prime)\}$ 
and $\{\psi_z^{(1)}(z,y^\prime),\dotsc,\psi_z^{(n^{\prime\prime})}(z,y^\prime)\}$ 
are linearly independent near $(z_0,x_0)$ and $(z_0,y_0)$ respectively since 
$\operatorname{rank}\bigl(\phi_z(z,x^\prime)\bigr)=n^{\prime\prime}$ and 
$\operatorname{rank}\bigl(\psi_z(z,y^\prime)\bigr)=n^{\prime\prime}$. 
Note that 
$$
\operatorname{Ker}\bigl(\psi_z(z_0,y_0^\prime)\bigr)
=
\operatorname{span}
\bigl\langle
\psi_z^{(1)}(z_0,y_0^\prime)^T,\dotsc,\psi_z^{(n^{\prime\prime})}(z_0,y_0^\prime)^T
\bigr\rangle^\perp, 
$$
where $\perp$ means the orthogonal complement in $\mathbb{R}^N$. 
We deduce that $\operatorname{dim}\bigl(V_{z_0}(x_0,y_0)\bigr)=n^{\prime\prime}-k$ is equivalent to 
\begin{equation}
\operatorname{dim}
\Bigl(
\operatorname{span}
\bigl\langle
\phi_z^{(1)}(z_0,x_0^\prime),\dotsc,\phi_z^{(n^{\prime\prime})}(z_0,x_0^\prime)
\bigr\rangle
\cap
\operatorname{span}
\bigl\langle
\psi_z^{(1)}(z_0,y_0^\prime),\dotsc,\psi_z^{(n^{\prime\prime})}(z_0,y_0^\prime)
\bigr\rangle
\Bigr)=k.
\label{equation:equiv1} 
\end{equation}
We express \eqref{equation:equiv1} by $k$ equations. 
\eqref{equation:equiv1} means that there exist linearly independent $k$ elements of 
\\ 
$\operatorname{span}\bigl\langle\phi_z^{(1)}(z_0,x_0^\prime),\dotsc,\phi_z^{(n^{\prime\prime})}(z_0,x_0^\prime)\bigr\rangle$ belonging to 
$\operatorname{span}\bigl\langle\psi_z^{(1)}(z_0,y_0^\prime),\dotsc,\psi_z^{(n^{\prime\prime})}(z_0,y_0^\prime)\bigr\rangle$. 
Denote by 
$\{\tilde{\psi}_z^{(1)}(z_0,y_0^\prime),\dotsc,\tilde{\psi}_z^{(n^{\prime\prime})}(z_0,y_0^\prime)\}$ 
the Schmidt orthonormalization of  
$\{\psi_z^{(1)}(z_0,y_0^\prime),\dotsc,\psi_z^{(n^{\prime\prime})}(z_0,y_0^\prime)\}$. 
We deduce that \eqref{equation:equiv1} is equivalent to the following: there exist 
$\lambda_1,\dotsc,\lambda_k \in \mathbb{R}^{n^{\prime\prime}}$, 
$\lambda_l=[\lambda_{l1},\dotsc,\lambda_{ln^{\prime\prime}}]$ ($l=1,\dotsc,k$) such that 
$\lambda_1,\dotsc,\lambda_k$ are linearly independent and if we set 
\begin{align*}
  \phi_z^{\lambda_l}(z,x^\prime)
& :=
  \lambda_l
  \phi_z(z,x^\prime)
  =
  \sum_{m=1}^{n^{\prime\prime}}
  \lambda_{lm}
  \phi_z^{(m)}(z,x^\prime),
\\
  H^{\lambda_l}(z,x^\prime,y^\prime)
& :=
  \phi_z^{\lambda_l}(z,x^\prime)\phi_z^{\lambda_l}(z,x^\prime)^T
  -
  \sum_{m=1}^{n^{\prime\prime}}
  \bigl\lvert
  \phi_z^{\lambda_l}(z,x^\prime)
  \tilde{\psi}_z^{(m)}(z,y^\prime)^T
  \bigr\rvert^2 
\\
& =
  \lambda_l\phi_z(z,x^\prime)
  \bigl(I_N-\tilde{\psi}_z^T(z,y^\prime)\tilde{\psi}_z(z,y^\prime)\bigr)
  \phi_z(z,x^\prime)^T\lambda_l^T, 
\\ 
  H^\lambda(z,x^\prime,y^\prime)
& :=
  [H^{\lambda_1}(z,x^\prime,y^\prime),\dotsc,H^{\lambda_k}(z,x^\prime,y^\prime)]^T, 
\end{align*}
for $l=1,\dotsc,k$, then 
\begin{equation}
H^{\lambda}(z_0,x_0^\prime,y_0^\prime)=0. 
\label{equation:equiv2}
\end{equation}
Here $I_N$ is the $N{\times}N$ identity matrix. 
\eqref{equation:equiv2} is the Plancherel-Parseval formula for 
$\phi_z^{\lambda_l}(z_0,x_0^\prime)$ ($l=1,\dotsc,k$) in the vector subspace spanned by 
$\tilde{\psi}_z^{(m)}(z_0,y_0^\prime)$ ($m=1,\dotsc,n^{\prime\prime}$). 
\par
We now introduce an artificial condition (H) for the regular conjugate triplets.  
\begin{description}
\item[Condition (H)] 
Let $k=1,\dotsc,n^{\prime\prime}$. Suppose that $(z_0;x_0,y_0) \in C_{R,k}$. 
Suppose that $H^{\lambda}(z,x^\prime,y^\prime)$ is the same as that of the previous paragraph and satisfies \eqref{equation:equiv2}. 
\\
Condition (H) is that $\operatorname{rank}\bigl(D_{z,x^\prime,y^\prime}H^{\lambda}(z_0,x_0^\prime,y_0^\prime)\bigr)=1$ holds for any choice of linearly independent $\lambda_1,\dotsc,\lambda_k \in \mathbb{R}^{n^{\prime\prime}}$.  
\end{description}
\par
We explain the basic idea behind Condition (H). Roughly speaking, the motivation comes from the geodesic X-ray transform with regular conjugate points. In particular, we clarify the meaning of the requirements “rank one” and “for any choice of $\lambda_1,\dotsc,\lambda_k \in \mathbb{R}^{n^{\prime\prime}}$
\begin{itemize}
\item 
{\bf rank one}: 
In the case of the geodesic X-ray transform, regular conjugate points correspond to a rank drop of exactly one for the differential of the exponential map. Condition (H) is designed so that our setting fits this geometric situation. Thus, the rank-one condition reflects the structure of regular conjugate points in the geodesic X-ray transform. 
\item 
{\bf for any choice of $\lambda_1,\dotsc,\lambda_k$}: 
This condition is imposed to exclude the possibility that two different choices 
$\lambda_1,\dotsc,\lambda_k$ 
and 
$\lambda_1^\prime,\dotsc,\lambda_k^\prime$ 
give rise to level sets
$$
\{H^{\lambda}(z,x^\prime,y^\prime)\bigr)=1\}, 
\quad
\{H^{\lambda^\prime}(z,x^\prime,y^\prime)\bigr)=1\}
$$
that intersect transversally. Such a situation leads to a union of submanifolds that is no longer smooth, making the microlocal analysis of the associated canonical relation significantly more complicated. To avoid these technical difficulties and to keep the exposition transparent, we impose this slightly stronger assumption. While it is not expected to be essential, removing it would require a substantially more involved analysis that is beyond the scope of this paper. 
\end{itemize}
\par
Under this condition, each $C_{R,k}$ becomes a submanifold of $\mathcal{G}{\times}X{\times}X$. 
This fact will play a crucial role in our intersection calculus for the composition 
$\mathcal{R}^{\ast}\mathcal{R}$ under the assumption $C_S=\emptyset$ later. 
More precisely we have the following. 
\begin{lemma}
\label{theorem:submanifold} 
Suppose that $Z$ is a double fibration and $N \geqq 2n^{\prime\prime}$. 
In addition we assume that the condition {\rm (H)} holds 
for all the regular conjugate triplets. 
Then for any $k=1,\dotsc,n^{\prime\prime}$, 
$C_{R,k}$ is an $(N+2n^\prime-1)$-dimensional embedded submanifold of $\mathcal{G}{\times}X{\times}X$. 
\end{lemma}
\begin{proof}
Fix arbitrary $k=1,\dotsc,n^{\prime\prime}$, 
and pick up arbitrary $(z_0;x_0,y_0) \in C_{R,k}$. 
Suppose that $Z$ is expressed as 
$\{x^{\prime\prime}=\phi(z,x^\prime)\}$ near $(z_0,x_0)$ 
and 
$\{y^{\prime\prime}=\psi(z,y^\prime)\}$ near $(z_0,y_0)$ 
respectively. 
Let $H^\lambda(z,x^\prime,y^\prime)$ be the same as 
that of the definition of Condition (H). 
Then $H^\lambda(z_0,x_0^\prime,y_0^\prime)=0$ and 
$\operatorname{rank}\bigl(D_{z,x^\prime,y^\prime}H^{\lambda}(z_0,x_0^\prime,y_0^\prime)\bigr)=1$. 
Pick up nonzero row 
$\nabla_{z,x^\prime,y^\prime}H^{\lambda_l}(z_0,x_0^\prime,y_0^\prime)$ of 
$D_{z,x^\prime,y^\prime}H^{\lambda}(z_0,x_0^\prime,y_0^\prime)$. 
The implicit function theorem implies that 
$\{H^{\lambda_l}(z,x^\prime,y^\prime)=0\}$ is a hypersurface in 
$\mathcal{G} \times X \times X$ near $(z_0;x_0,y_0)$. 
Then $\{H^{\lambda}(z,x^\prime,y^\prime)=0\}$ is also a hypersurface in 
$\mathcal{G} \times X \times X$ near $(z_0;x_0,y_0)$ 
since $\operatorname{rank}\bigl(D_{z,x^\prime,y^\prime}H^{\lambda}(z_0,x_0^\prime,y_0^\prime)\bigr)=1$. Therefore 
$$
\{x^{\prime\prime}=\phi(z,x^\prime)\}
\cap
\{y^{\prime\prime}=\psi(z,y^\prime)\}
\cap
\{H^{\lambda}(z,x^\prime,y^\prime)=0\}
\subset 
C_{R,k}
$$
near $(z_0;x_0,y_0)$ since 
$(z_0;x_0,y_0)$ is a regular $Z$-conjugate triplet of degree $k$. 
Using the Condition (H) again, we deduce that 
the connected component of $C_{R,k}$ containing $(z_0;x_0,y_0)$ is characterized by 
$$
F(x^{\prime\prime},y^{\prime\prime},z;x^\prime,y^\prime)
:=
\begin{bmatrix}
x^{\prime\prime}-\phi(z,x^\prime)
\\
y^{\prime\prime}-\psi(z,y^\prime) 
\\
H^\lambda(z,x^\prime,y^\prime)
\end{bmatrix}
=0
$$
near $(z_0;x_0,y_0)$. 
The differential of $F$ at $(z_0;x_0,y_0)$ is 
\begin{align*}
& DF(x_0^{\prime\prime},y_0^{\prime\prime},z_0;x_0^\prime,y_0^\prime)
\\
  =
& \begin{bmatrix}
  I_{n^{\prime\prime}} & O & -\phi_z(z_0,x_0^\prime) & -\phi_{x^\prime}(z_0,x_0^\prime) & O
  \\
  O & I_{n^{\prime\prime}} &  -\psi_z(z_0,x_0^\prime) & O & -\psi_{x^\prime}(z_0,y_0^\prime)
  \\
  O & O & H_z^\lambda(z_0,x_0^\prime,y_0^\prime) & H_{x^\prime}^\lambda(z_0,x_0^\prime,y_0^\prime) & H_{y^\prime}^\lambda(z_0,x_0^\prime,y_0^\prime)
  \end{bmatrix}. 
\end{align*}
Define an upper triangular matrix $Q$ by 
$$
Q(z,x^\prime,y^\prime)
:=
\begin{bmatrix}
I_{n^{\prime\prime}} & O & \phi_z(z,x^\prime) & \phi_{x^\prime}(z,x^\prime) & O 
\\
  & I_{n^{\prime\prime}} & \psi_z(z,y^\prime) & O & \phi_{x^\prime}(z,y^\prime) 
\\
  &                      & I_N                & O            & O
\\
  &                      &                    & I_{n^\prime} & O
\\
  &                      &                    &              & I_{n^\prime}
\end{bmatrix}.
$$
Then $\operatorname{det}\bigl(Q(z,x^\prime,y^\prime)\bigr)\equiv1$ and 
$$
DF(x_0^{\prime\prime},y_0^{\prime\prime},z_0;x_0^\prime,y_0^\prime)
Q(z_0,x_0^\prime,y_0^\prime)
=
\begin{bmatrix}
I_{2n^{\prime\prime}} & O
\\
O & D_{z,x^\prime,y^\prime}H^\lambda(z_0,x_0^\prime,y_0^\prime) 
\end{bmatrix}. 
$$
Hence $\operatorname{rank}\bigl(DF(x_0^{\prime\prime},y_0^{\prime\prime},z_0;x_0^\prime,y_0^\prime)\bigr)=2n^{\prime\prime}+1$, and we deduce that $C_{R,k}$ is an embedded 
submanifold of $\mathcal{G} \times X \times X$ of dimension 
$(N+2n)-(2n^{\prime\prime}+1)=N+2n^\prime-1$. 
\end{proof}
%
%
\section{Normal operators without conjugate points}
\label{section:bolker}
We begin with the definition of the Bolker condition,  
and review the relationship between the Bolker condition and $Z$-conjugacy quickly 
following \cite{MazzucchelliSaloTzou}. 
\begin{definition}
\label{theorem:bolker} 
Suppose that $Z$ is a double fibration. 
We say that the canonical relation $(N^\ast{Z}\setminus0)^\prime$ satisfies the Bolker condition at 
$(z,\zeta,x,\eta) \in (N^\ast{Z}\setminus0)^\prime$ 
if $\pi_\text{L}$ is an injective immersion at $(z,\zeta,x,\eta)$, 
that is, 
$$
\pi_\text{L}^{-1}\bigl((z,\zeta)\bigr)
=
\{(z,\zeta,x,\eta)\},
\quad
\operatorname{rank}
\bigl(D\pi_\text{L}\vert_{(z,\zeta,x,\eta)}\bigr)
=
N+n.
$$
We say that the canonical relation $(N^\ast{Z}\setminus0)^\prime$ satisfies the Bolker condition if 
$\pi_\text{L}$ is an injective immersion everywhere in $(N^\ast{Z}\setminus0)^\prime$.  
\end{definition}
The injectivities of $\pi_\text{L}$ and $D\pi_\text{L}$ were respectively characterized by 
Mazzucchelli, Salo and Tzou in \cite{MazzucchelliSaloTzou}. 
\begin{lemma}[{Injectivity of $\pi_\text{L}$, \cite[{Lemmas~3.1 and 3.2}]{MazzucchelliSaloTzou}}]
\label{theorem:injectivity1} 
Suppose that $Z$ is a double fibration. 
Then $\pi_\text{L}^{-1}\bigl((z,\zeta)\bigr)=\{(z,\zeta,x,\eta)\}$ holds if and only if 
$\eta\bigl(V_z(x,y)\bigr)\ne\{0\}$ for any $y \in G_z\setminus\{x\}$. 
In particular, $\pi_\text{L}^{-1}\bigl((z,\zeta)\bigr)=\{(z,\zeta,x,\eta)\}$ holds if and only if  
$(z;x,y)$ is not a Z-conjugate triplet for any $y \in G_z\setminus\{x\}$. 
Moreover $\pi_\text{L}$ is injective on $Z$ if and only if there are no $Z$-conjugate triplets 
in $\mathcal{G}{\times}X{\times}X$. 
\end{lemma}
\begin{lemma}[{Injectivity of $D\pi_\text{L}$, \cite[{Lemma~3.3}]{MazzucchelliSaloTzou}}]
\label{theorem:injectivity2} 
Let $Z$ be a double fibration and let $(z_0,\zeta_0,x_0,\eta_0) \in (N^\ast{Z}\setminus0)^\prime$. 
Suppose that $Z$ is given by 
$\{x^{\prime\prime}=\phi(z,x^\prime)\}$ 
and 
$\{z^{\prime\prime}=b(x,z^\prime)\}$ 
near $(z_0,x_0)$. 
The following conditions are mutually equivalent.
\begin{itemize}
\item[(a)] 
$D\pi_\text{L}\vert_{(z_0,\zeta_0,x_0,\eta_0)}$ is injective. 
\item[(b)] 
$D\pi_\text{R}\vert_{(z_0,\zeta_0,x_0,\eta_0)}$ is surjective.   
\item[(c)] 
$\operatorname{rank}\bigl[\phi_z(z_0,x_0^\prime)^T,\partial_{x^\prime}\bigl(\phi_z(z_0,x^\prime)^T\eta_0^{\prime\prime}\bigr)\vert_{x^\prime=x_0^\prime}\bigr]=n$.  
\item[(d)] 
$\operatorname{rank}\bigl[b_x(x_0,z_0^\prime)^T,\partial_{z^\prime}\bigl(b_x(x_0,z^\prime)^T\zeta_0^{\prime\prime}\bigr)\vert_{z^\prime=z_0^\prime}\bigr]=n$.  
\end{itemize}
\end{lemma}
It is worth to mention that if $D\pi_\text{L}\vert_{(z_0,\zeta_0,x_0,\eta_0)}$ is injective, 
then so is $\pi_\text{L}$ near $(z_0,\zeta_0,x_0,\eta_0)$. 
See \cite[{Lemma~3.4}]{MazzucchelliSaloTzou} for the detail. 
\par
We now state our first results about $\mathcal{R}^\ast\mathcal{R}$ 
when there are no $Z$-conjugate triplets. 
This is not an essentially new result. See, e.g., \cite{Quinto}. 
But we state and prove the theorem 
since we make use of the similar arguments in the next section. 
The explicit formula of $\mathcal{R}^\ast\mathcal{R}$ is given by 
\begin{equation}
\mathcal{R}^\ast\mathcal{R}f(x)
=
\left(
\iint_{H_x{\times}G_z}
\overline{\kappa(z,x)}
\kappa(z,y)
\frac{f}{\lvert{dX}\rvert^{1/2}}(x)
dG_z(y)dH_x(z) 
\right)
\lvert{dX(x)}\rvert^{1/2}, 
\quad
x \in X,
\label{equation:normal}
\end{equation}
for $f \in \mathscr{D}(X,\Omega_X^{1/2})$. 
Set 
$$
\mathcal{C}:=(N^\ast{Z}\setminus0)^\prime,
\quad  
\Delta(T^\ast\mathcal{G}):=\{(z,\zeta,z,\zeta) : (z,\zeta) \in T^\ast\mathcal{G}\}
$$ 
$$
\Delta:=T^\ast{X}\times\Delta(T^\ast\mathcal{G}){\times}T^\ast{X},
\quad
E:=(\mathcal{C}^T\times\mathcal{C})\cap\Delta
$$
for short. The projection $\pi_E$ is defined by 
$$
\pi_E: 
E\ni(x,\eta,z,\zeta,z,\zeta,y,\tilde{\eta}) 
\mapsto 
(x,\eta,y,\tilde{\eta})
\in 
T^\ast{X}{\times}T^\ast{X}. 
$$
\begin{theorem}
\label{theorem:normal1}
Suppose that $Z$ is a double fibration. 
In addition, we assume the following conditions.
\begin{itemize}
\item 
$N \geqq 2n^{\prime\prime}$. 
\item 
$\pi_X: Z \rightarrow X$ is proper, and $\pi_X^{-1}(x)$ is connected for any $x \in X$. 
\item 
There are no $Z$-conjugate triplets, 
and $D\pi_\text{L}$ is injective at all $(z,\zeta,x,\eta) \in \mathcal{C}$. 
\end{itemize}
Then $\mathcal{C}^T\circ\mathcal{C}$ is a clean intersection with excess $e=N-n$, 
and $\mathcal{R}^\ast\mathcal{R}$ is an elliptic pseudodifferential operator of order $-n^\prime$ on $X$.   
\end{theorem}
\begin{proof}
Lemma~11 implies that $\pi_L$ is injective since we assume that there are no $Z$-conjugate triplets. 
We also assume that $D\pi_L$ is injective at any point of $\mathcal{C}$. Then the Bolker condition is satisfied on $\mathcal{C}$. It suffices to show the following. 
\begin{itemize}
\item 
$\pi_E(E)=\Delta(T^\ast{X}\setminus0)$, 
where 
$\Delta(T^\ast{X}\setminus0)=\{(x,\eta,x,\eta) : (x,\eta) \in T^\ast{X}\setminus0\}$. 
\item 
$e:=\operatorname{codim}(\mathcal{C}^T\times\mathcal{C})+\operatorname{codim}\bigl(\Delta(T^\ast\mathcal{G})\bigr)-\operatorname{codim}(E)=N-n$, and the order of $\mathcal{R}^\ast\mathcal{R}$ is $-n^\prime$. 
\item 
$\pi_E$ is a proper mapping, and 
$\pi_E^{-1}\bigl((x,\eta,x,\eta)\bigr)$ is connected. 
\item 
For any $c_0 \in E$, if $\Xi \in T_{c_0}(\mathcal{C}^T\times\mathcal{C}){\cap}T_{c_0}\Delta$, 
then $\Xi \in T_{c_0}E$.  
\end{itemize}
If these statements are proved, then we can deduce that $\mathcal{R}^\ast\mathcal{R}$ is a pseudodifferential operator of order $-n^\prime$ on $X$ using \cite[{Theorem~25.2.3}]{Hoermander4} and further the ellipticity of $\mathcal{R}^\ast\mathcal{R}$ follows since $\lvert\kappa(z,x)\rvert^2$ is nowhere vanishing. 
\par
We have $\pi_E(E)=\Delta(T^\ast{X}\setminus0)$ 
using the injectivity of 
$\pi_\text{L}$ of the Bolker condition. 
We compute the excess $e$. Note that 
\begin{align*}
  \operatorname{codim}(\Delta)
& =
  \dim(T^\ast\mathcal{G})=2N, 
\\
  \dim(\mathcal{C})
& =
  \dim(\mathcal{G}{\times}X)=N+n,
\\
  \dim(\mathcal{C}^T\times\mathcal{C})
& =
  2\dim(\mathcal{C})=2N+2n. 
\end{align*}
We use the injectivity of $\pi_\text{L}$ of the Bolker condition again to deduce 
$$
\dim(E)=\dim(\mathcal{C})=N+n. 
$$
Then we obtain 
\begin{align*}
  e
& =
  \operatorname{codim}(\mathcal{C}^T\times\mathcal{C})
  +
  \operatorname{codim}\bigl(\Delta(T^\ast\mathcal{G})\bigr)
  -
  \operatorname{codim}(E)
\\
& =
  -
  \operatorname{dim}(\mathcal{C}^T\times\mathcal{C})
  +
  \operatorname{codim}\bigl(\Delta(T^\ast\mathcal{G})\bigr)
  +
  \operatorname{dim}(E)
\\
& =
  -(2N+2n)+(2N)+(N+n)
  =
  N-n, 
\end{align*}
and the order of $\mathcal{R}^\ast\mathcal{R}$ is given by 
$$
\text{the order of $\mathcal{R}^\ast$}\ 
+\ 
\text{the order of $\mathcal{R}$}\ 
+\
\frac{e}{2}
=
-\frac{N+2n^\prime-n}{2}+\frac{N-n}{2}
=
-n^\prime. 
$$
\par
We check the required properties of $\pi_E$. 
We use the linear mapping $A(z,x)$ and the injectivity of $\pi_\text{L}$ to have 
$$
\pi_E: 
E 
\ni
(x,\eta,z,A(z,x)\eta,z,A(z,x)\eta,x,\eta)
\mapsto 
(x,\eta,x,\eta)
\in 
T^\ast{X}{\times}T^\ast{X}. 
$$
Then we have 
$$
\pi_E^{-1}\bigl((x,\eta,x,\eta)\bigr)
=
\{
(x,\eta,z,A(z,x)\eta,z,A(z,x)\eta,x,\eta) 
: 
z \in H_x=\pi_\mathcal{G}\circ\pi_X^{-1}(x)
\}. 
$$
We deduce that $\pi_E$ is proper since $\pi_X$ is proper, 
and that $\pi_E^{-1}\bigl((x,\eta,x,\eta)\bigr)$ is connected 
since $\pi_X^{-1}(x)$ is connected and $A(z,x)$ is a linear mapping depending smoothly on $(z,x)$. 
\par
Finally we show that $E=(\mathcal{C}^T\times\mathcal{C})\cap\Delta$ is a clean intersection. 
Fix arbitrary $c_0 \in E$, and pick up arbitrary 
$\Xi=(\Xi_1,\Xi_2,\Xi_3,\Xi_4) \in T_{c_0}(\mathcal{C}^T\times\mathcal{C})$. 
Consider a smooth curve in $\mathcal{C}^T\times\mathcal{C}$ of the form 
$$
c(s)
=
c_0+s\Xi+\mathcal{O}(s^2)
=
c_0
+
s(\Xi_1,\Xi_2,\Xi_3,\Xi_4)
+
\mathcal{O}(s^2)
\quad\text{near}\quad
s=0.
$$
If $\Xi \in T_{c_0}\Delta$, then $\Xi_2=\Xi_3$, 
and moreover the injectivity of $D\pi_\text{L}$ implies that $\Xi_1=\Xi_4$. 
Thus we obtain $\Xi=(\Xi_1,\Xi_2,\Xi_2,\Xi_1) \in T_{c_0}E$. 
This completes the proof. 
\end{proof}
%
%
\section{Normal operators with conjugate points}
\label{section:main}
Finally in this section we study the normal operator $\mathcal{R}^\ast\mathcal{R}$ 
with regular $Z$-conjugate triplets. 
Suppose that $Z$ is a double fibration, and split $C_{R,k}$ into the disjoint union of connected components:
$$
C_{R,k}:=\bigcup_{\alpha\in\Lambda_k}C_{R,k,\alpha},
$$
where $k=1,\dotsc,n^{\prime\prime}$ and 
$\Lambda_k$ ($k=1,\dotsc,n^{\prime\prime}$) are the sets of indices of connected components. 
We now remark that $C_{R,k}$ ($k=1,\dotsc,n^{\prime\prime}$) are disjoint and then 
so are $C_{R,k,\alpha}$ ($k=1,\dotsc,n^{\prime\prime}$, $\alpha \in \Lambda_k$).    
We now state the main theorem of this paper. 
\begin{theorem}
\label{theorem:normal2}
Suppose that $Z$ is a double fibration. 
In addition we assume the following conditions. 
\begin{itemize}
\item 
$C_S=\emptyset$.
\item 
$N \geqq 2n^{\prime\prime}$. 
\item 
Condition {\rm (H)}.
\item
$\pi_X$ is proper, and $\pi_X^{-1}(x)$ is connected for any $x \in X$. 
\item 
If $\pi_\text{L}^{-1}\bigl((z,\zeta)\bigr)=\{(z,\zeta,x,\eta)\}$ for $(z,\zeta,x,\eta) \in \mathcal{C}$, then $D\pi_\text{L}\vert_{(z,\zeta,x,\eta)}$ is injective. 
\end{itemize}
Then we have a decomposition of $\mathcal{R}^\ast\mathcal{R}$ of the form 
\begin{equation}
\mathcal{R}^\ast\mathcal{R}
=
P
+
\sum_{k=1}^{n^{\prime\prime}}
\sum_{\alpha \in \Lambda_k}
A_{k,\alpha},
\label{equation:ragasa}
\end{equation}
where $P$ is an elliptic pseudodifferential operator of order $-n^\prime$ on $X$, 
$A_{k,\alpha}$ is a Fourier integral operator whose distribution kernel belongs to 
$$
\mathcal{I}^{-(n+1-k)/2}
(X{\times}X,\mathcal{C}_{A_{k,\alpha}}^\prime;\Omega_{X{\times}X}^{1/2}),
$$
and the canonical relation of $A_{k,\alpha}$ is given by 
\begin{align}
  \mathcal{C}_{A_{k,\alpha}}
& =
  \{
  (x,\eta;y,\tilde{\eta}) \in T^\ast{X}\setminus0 \times T^\ast{X}\setminus0 : 
  \exists(z,\zeta) \in N^\ast{H_x}{\cap}N^\ast{H_y}
  \ \text{s.t.}
\nonumber
\\ 
& \qquad
  (z;x,y) \in C_{R,k,\alpha},\ 
  \eta=B(z,x)\zeta,\  
  \tilde{\eta}=B(z,y)\zeta\}.   
\label{equation:canonicalrelation}
\end{align}
\end{theorem}
The highest order part of the right hand side of \eqref{equation:ragasa} is 
the FIO part of $k=n^{\prime\prime}$
$$
\sum_{\alpha \in \Lambda_{n^{\prime\prime}}}
A_{n^{\prime\prime},\alpha},
\quad
A_{n^{\prime\prime},\alpha}
\in
\mathcal{I}^{-(n^\prime+1)/2}
(X{\times}X,\mathcal{C}_{A_{n^{\prime\prime},\alpha}}^\prime;\Omega_{X{\times}X}^{1/2}).
$$
If $n^\prime=1$, then $-n^\prime=-(n^\prime+1)/2$, otherwise 
$-n^\prime<-(n^\prime+1)/2$. 
This could be an obstruction for the invertibility of $\mathcal{R}$. 
\begin{proof}[Proof of Theorem~\ref{theorem:normal2}]
Set $C_1:=\mathcal{G}{\times}X{\times}X\setminus{C_R}$. 
Then all the elements of $C_1$ are not $Z$-conjugate triplets since $C_S=\emptyset$. 
Set 
$$
Z_1:=\{(z,x) \in Z : (z,x,y) \in C_1 \ \text{with some} \ y \in X\},
$$
$$
\mathcal{N}_1
:=
N^\ast{Z}\setminus0\vert_{Z_1}
=
\{(z,\zeta,x,\eta) : (z,x) \in Z_1\}.
$$
Lemma~\ref{theorem:injectivity1} implies that $\pi_L$ is injective on $\mathcal{N}_1$, 
and $D\pi_L$ is injective at any point of $\mathcal{N}_1$ from the assumption.  
Then the Bolker condition is satisfied on $\mathcal{N}_1$. 
\par
The strategy of the proof is basically due to Holman and Uhlmann \cite{HolmanUhlmann}. 
Set $C_\delta:=\{(z;x,x) : z \in \mathcal{G}, x \in X\}$. 
Since $C_S=\emptyset$ and 
$C_{R,k,\alpha}$ ($k=1,\dotsc,n^{\prime\prime}$, $\alpha \in \Lambda_k$) are disjoint, 
we deduce that $C_\delta$ and $C_{R,k,\alpha}$ are mutually disjoint for any 
$k=1,\dotsc,n^{\prime\prime}$ and $\alpha \in \Lambda_k$. 
Furthermore we deduce the following:
\begin{itemize}
\item 
There exist an open neighborhood $U_\delta$ of $C_\delta$ in $\mathcal{G}{\times}X{\times}X$ and 
open neighborhoods $U_{k,\alpha}$ of $C_{k,\alpha}$ in $\mathcal{G}{\times}X{\times}X$ such that 
$\overline{U_\delta}$ and all the $\overline{U_{k,\alpha}}$ are mutually disjoint. 
\item 
There exist compactly supported smooth functions 
$\varphi_\delta(z,x,y)$ and $\varphi_{k,\alpha}(z,x,y)$ 
on $\mathcal{G}{\times}X{\times}X$ such that  
\begin{alignat*}{3}
  0
& \leqq 
  \varphi_\delta
  \leqq
  1,
& \quad
  \varphi_\delta
& =
  1 
  \quad\text{on}\quad
  C_\delta,
& \quad
  \operatorname{supp}[\varphi_\delta]
& \subset
  U_\delta,
\\
  0
& \leqq 
  \varphi_{k,\alpha}
  \leqq
  1,
& \quad
  \varphi_{k,\alpha}
& =
  1 
  \quad\text{on}\quad
  C_{k,\alpha},
& \quad
  \operatorname{supp}[\varphi_\delta]
& \subset
  U_{k,\alpha}. 
\end{alignat*}
\item 
If we set 
$$
\varphi_0(z,x,y)
:=
1
-
\varphi_\delta(z,x,y)
-
\sum_{k=1}^{n^{\prime\prime}}
\sum_{\alpha \in \Delta_k}
\varphi_{k,\alpha}(z,x,y),
$$
then 
$$
0 \leqq \varphi_0 \leqq 1,
\qquad 
\varphi_0=1
\quad\text{on}\quad
(\mathcal{G}{\times}X{\times}X)\setminus \left(U_\delta \bigcup \left(\displaystyle\cup_{k=1}^{n^{\prime\prime}}\displaystyle\cup_{\alpha\in\Lambda_k}U_{k,\alpha}\right)\right), 
$$
$$
U_0
:=
\operatorname{supp}[\varphi_0]
\subset
(\mathcal{G}{\times}X{\times}X)\setminus \left(C_\delta \bigcup \left(\displaystyle\cup_{k=1}^{n^{\prime\prime}}\displaystyle\cup_{\alpha\in\Lambda_k}C_{R,k,\alpha}\right)\right).
$$
\end{itemize}
The set of nonnegative smooth functions 
$\{\varphi_0,\varphi_\delta\}\cup\{\varphi_{k,\alpha}\}$ 
forms a partition of unity on $\mathcal{G}{\times}X{\times}X$. 
We substitute this into \eqref{equation:normal} to obtain 
$$
\mathcal{R}^\ast\mathcal{R}f(x)
=
A_0f(x)+A_\delta{f}(x)
+
\sum_{k=1}^{n^{\prime\prime}}
\sum_{\alpha \in \Lambda_k}
A_{k,\alpha}f(x),  
$$
\begin{align}
  \frac{A_0 f}{\lvert{dX}\rvert^{1/2}}(x)
& =
  \iint_{H_x{\times}G_z}
  \varphi_0(z,x,y)
  \overline{\kappa(z,x)}\kappa(z,y)
  \frac{f}{\lvert{dX}\rvert^{1/2}}(x)
  dG_z(y)
  dH_x(z),
\nonumber
\\
  \frac{A_\delta f}{\lvert{dX}\rvert^{1/2}}(x)
& =
  \iint_{H_x{\times}G_z}
  \varphi_\delta(z,x,y)
  \overline{\kappa(z,x)}\kappa(z,y)
  \frac{f}{\lvert{dX}\rvert^{1/2}}(x)
  dG_z(y)
  dH_x(z),
\nonumber
\\ 
  \frac{A_{k,\alpha} f}{\lvert{dX}\rvert^{1/2}}(x)
& =
  \iint_{H_x{\times}G_z}
  \varphi_{k,\alpha}(z,x,y)
  \overline{\kappa(z,x)}\kappa(z,y)
  \frac{f}{\lvert{dX}\rvert^{1/2}}(x)
  dG_z(y)
  dH_x(z),
\label{equation:target}
\end{align}
for $f \in \mathscr{D}(X,\Omega_X^{1/2})$. 
We remark that 
$\operatorname{supp}[\varphi_0]\cup\operatorname{supp}[\varphi_\delta]$ 
do not contain $Z$-conjugate triplets, 
and that the injectivity of $\pi_L$ holds at all the points given by 
\begin{equation}
\bigl(z,\zeta,x,B(z,x)\zeta\bigr), 
\bigl(z,\zeta,y,B(z,y)\zeta\bigr) 
\in 
\mathcal{C}, 
\quad
(z,x,y) \in \operatorname{supp}[\varphi_0]\cup\operatorname{supp}[\varphi_\delta]. 
\label{equation:points}
\end{equation}
Hence the assumption on $D\pi_L$ ensures 
that the Bolker condition holds all the points given by \eqref{equation:points}. 
Using the assumptions on $\pi_X$, 
we deduce that $A_\delta$ is an elliptic pseudodifferential operator of order $-n^\prime$ in the same way as the proof of Theorem~\ref{theorem:normal1}, 
and that $A_0$ is a pseudodifferential operator of order $-\infty$. 
Thus we set $P:=A_\delta+A_0$. 
\par
In what follows we concentrate on the analysis of $A_{k,\alpha}$. 
We first prepare the intersection calculus for $A_{k,\alpha}$. 
We remark that the contribution of $(z;x,y) \in U_{k,\alpha}\setminus{C_{R,k,\alpha}}$ 
in \eqref{equation:target} can be contained in $A_0$, and we have only to consider 
$(z;x,y) \in C_{R,k,\alpha}$ in \eqref{equation:target}. 
The canonical relation $\mathcal{C}_{A_{k,\alpha}}$ is given by 
\eqref{equation:canonicalrelation} as a result of the computation of the composition 
$\mathcal{C}^T\circ\mathcal{C}$ with the restriction $(z;x,y) \in C_{R,k,\alpha}$.  
\par
We now define a subset $D_{k,\alpha}$ and $E_{k,\alpha}$ of 
$T^\ast{X}{\times}T^\ast\mathcal{G}{\times}T^\ast\mathcal{G}{\times}T^\ast{X}$ by 
\begin{align}
  D_{k,\alpha}
& :=
  \{(x,\eta,z,\zeta;\tilde{z},\tilde{\zeta},y,\tilde{\eta})
  \in \mathcal{C}^T \times \mathcal{C}
  : 
\nonumber
\\
& \qquad
  \exists y_1, \exists x_1 \in X
  \ \text{s.t.}\ 
  (z;x,y_1), (\tilde{z};x_1,y) \in C_{R,k,\alpha}
  \} 
\nonumber
\\
& =
  \bigl\{
  \bigl(x,B(z,x)\zeta,z,\zeta;\tilde{z},\tilde{\zeta},y,B(\tilde{z},y)\tilde{\zeta}\bigr)
  :
  \zeta \in N^\ast_zH_x\setminus\{0\}, \tilde{\zeta} \in N^\ast_{\tilde{z}}H_y\setminus\{0\},
\nonumber
\\
& \qquad
  \exists y_1, \exists x_1 \in X
  \ \text{s.t.}\ 
  (z;x,y_1), (\tilde{z};x_1,y) \in C_{R,k,\alpha}
  \bigr\}, 
\nonumber
\\
  E_{k,\alpha}
& :=
  D_{k,\alpha}\cap\Delta
\nonumber
\\
& =
  \bigl\{
  \bigl(x,B(z,x)\zeta,z,\zeta;z,\zeta,y,B(z,y)\zeta\bigr)
  :
  \zeta \in (N^\ast_zH_x{\cap}N^\ast_zH_y)\setminus\{0\},
\nonumber
\\
& \qquad
  \exists y_1, \exists x_1 \in X
  \ \text{s.t.}\ 
  (z;x,y_1), (z;x_1,y) \in C_{R,k,\alpha}
  \bigr\}, 
\label{equation:delta}
\end{align}
where $\Delta=T^\ast{X}\times\Delta(T^\ast\mathcal{G}){\times}T^\ast{X}$ 
was defined in the previous section.  
The projection $\pi_{E_{k,\alpha}}$ is defined by 
$$
\pi_{E_{k,\alpha}}: 
E_{k,\alpha}\ni(x,\eta,z,\zeta,z,\zeta,y,\tilde{\eta}) 
\mapsto 
(x,\eta,y,\tilde{\eta})
\in 
T^\ast{X}{\times}T^\ast{X}. 
$$
Then we have $\pi_{E_{k,\alpha}}(E_{k,\alpha})=\mathcal{C}_{A_{k,\alpha}}$. 
\par
To complete the proof of Theorem~\ref{theorem:normal2}, we have only to show the following:  
\begin{itemize}
\item 
$e:=\operatorname{codim}(D_{k,\alpha})+\operatorname{codim}\bigl(\Delta(T^\ast\mathcal{G})\bigr)-\operatorname{codim}(E_{k,\alpha})=N-2n+n^\prime+k$, and the order of $A_{k,\alpha}$ is $-(n+1-k)/2$. 
\item 
$\pi_{E_{k,\alpha}}$ is a proper mapping, and 
$\pi_{E_{k,\alpha}}^{-1}\bigl((x,\eta,x,\eta)\bigr)$ is connected. 
\item 
For any $c_0 \in E_{k,\alpha}$, 
if $\Xi \in T_{c_0}(D_{k,\alpha}){\cap}T_{c_0}\Delta$, 
then $\Xi \in T_{c_0}E_{k,\alpha}$.  
\end{itemize}
\par
We compute the excess $e$. Note that 
\begin{align*}
  \dim(D_{k,\alpha})
& =
  2\dim(N^\ast{C_{R,k,\alpha}})
  =
  2\dim(N^\ast{Z})
  =
  2N+2n,
\\
  \operatorname{codim}(\Delta)
& =
  \dim(T^\ast\mathcal{G})
  =
  2N,
\\
  \dim(E_{k,\alpha})
& =
  \dim(C_{R,k,\alpha})
  +
  \dim(N^\ast_zH_x{\cap}N^\ast_zH_y)
  =
  N+2n^\prime-1+k. 
\end{align*}
Then we obtain 
\begin{align*}
  e
& =
  \operatorname{codim}(D_{k,\alpha})
  +
  \operatorname{codim}\bigl(\Delta(T^\ast\mathcal{G})\bigr)
  -
  \operatorname{codim}(E_{k,\alpha})
\\
& =
  -
  \operatorname{dim}(D_{k,\alpha})
  +
  \operatorname{codim}\bigl(\Delta(T^\ast\mathcal{G})\bigr)
  +
  \operatorname{dim}(E_{k,\alpha})
\\
& =
  -(2N+2n)+(2N)+(N+2n^\prime-1+k)
  =
  N-2n^{\prime\prime}-1+k.  
\end{align*}
We now remark that 
$$
e=N-2n^{\prime\prime}+(k-1) \geqq N-2n^{\prime\prime}\geqq0.
$$
The order of $A_{k,\alpha}$ is given by 
\begin{align*}
& \text{the order of $\mathcal{R}^\ast$}\ 
  +\ 
  \text{the order of $\mathcal{R}$}\ 
  +\
  \frac{e}{2}
\\
  =
& \frac{-N-2n^\prime+n}{2}+\frac{N-2n^{\prime\prime}-1+k}{2}
  =
  -
  \frac{n+1-k}{2}. 
\end{align*}
\par
We check the required properties of $\pi_{E_{k,\alpha}}$. 
We use the linear mapping $A(z,x)$ to have 
$$
\pi_{E_{k,\alpha}}: 
E 
\ni
(x,\eta,z,A(z,x)\eta,z,A(z,y)\tilde{\eta},y,\tilde{\eta})
\mapsto 
(x,\eta,y,\tilde{\eta})
\in 
T^\ast{X}{\times}T^\ast{X} 
$$
with $A(z,x)\eta=A(z,y)\tilde{\eta}$. 
Then we have 
$$
\pi_{E_{k,\alpha}}^{-1}\bigl((x,\eta,y,\tilde{\eta})\bigr)
=
\{
(x,\eta,z,A(z,x)\eta,z,A(z,y)\tilde{\eta},y,\tilde{\eta}) 
: 
z \in H_x{\cap}H_y=\pi_\mathcal{G}(\pi_X^{-1}(x)\cap\pi_X^{-1}(y))
\}. 
$$
We deduce that $\pi_{E_{k,\alpha}}$ is proper since $\pi_X$ is proper, 
and that $\pi_{E_{k,\alpha}}^{-1}\bigl((x,\eta,y,\tilde{\eta})\bigr)$ is connected 
since $\pi_X^{-1}(x)\cap\pi_X^{-1}(y)$ is connected and 
$A(z,x)$ and $A(z,y)$ are linear mappings depending smoothly on $(z,x)$ and $(z,y)$ respectively. 
\par
Finally we show that $E_{k,\alpha}=D_{k,\alpha}\cap\Delta$ is a clean intersection. 
Fix arbitrary $c_0 \in E_{k,\alpha}$, and pick up arbitrary 
$\Xi=(\Xi_1,\Xi_2,\Xi_3,\Xi_4) \in T_{c_0}D_{k,\alpha}$. 
Consider a smooth curve in $D_{k,\alpha}$ of the form 
\begin{align*}
    c(s)
& =
    (c_1(s),c_2(s),c_3(s),c_4(s))
    =
    c_0+s\Xi+\mathcal{O}(s^2)
\\
&  =
    (c_{0,1},c_{0,2},c_{0,3},c_{0,4})
    +
    s(\Xi_1,\Xi_2,\Xi_3,\Xi_4)
    +
    \mathcal{O}(s^2)
    \quad\text{near}\quad
    s=0.
\end{align*}
Note that $c_{0,2}=c_{0,3}$ since $c_0 \in E_{k,\alpha}$.  
Suppose that $\Xi \in T_{c_0}\Delta$. 
It suffices to show that  $\Xi=(\Xi_1,\Xi_2,\Xi_3,\Xi_4) \in T_{c_0}E_{k,\alpha}$. 
We have $(\Xi_1,\Xi_2) \in T_{(c_{0,1},c_{0,2})} \in \mathcal{C}^T$ and 
$(\Xi_3,\Xi_4) \in T_{(c_{0,2},c_{0,4})} \in \mathcal{C}$ 
since $c(s)$ is a smooth curve in $D_{k,\alpha}$ near $c_0$. 
It follows  that $(\Xi_2,\Xi_3) \in T_{(c_{0,2},c_{0,2})}\bigl(\Delta(T^\ast\mathcal{G}))\bigr)$ and $\Xi_2=\Xi_3$ to prove $\Xi=(\Xi_1,\Xi_2,\Xi_2,\Xi_4) \in T_{c_0}E_{k,\alpha}$, where $\Delta(T^\ast\mathcal{G})$ is the diagonal part of $T^\ast\mathcal{G}{\times}T^\ast\mathcal{G}$. 
\par
We argue the above more concretely. 
Recall Lemma~\ref{theorem:localexpression} and \eqref{equation:delta}. 
Then there exist $z_0 \in \mathcal{G}$, $x_0,y_0 \in G_{z_0}$ with $x_0{\ne}y_0$, 
$\zeta_0 \in (N^\ast_{z_0}H_{x_0}{\cap}N^\ast_{z_0}H_{y_0})\setminus\{0\}$, 
and functions $b(x,z^\prime)$ and $\tilde{b}(y,z^\prime)$ such that 
$G_z$ is given by $\{z^{\prime\prime}=b(x,z^\prime)\}$ near $(z_0,x_0)$, 
$G_z$ is given by $\{z^{\prime\prime}=\tilde{b}(y,z^\prime)\}$ near $(z_0,y_0)$, 
$$
c_{0,1}=\bigl(x_0,-b_x(x_0,z_0)^T\zeta_0^{\prime\prime}\bigr), 
\quad
c_{0,2}=c_{0,3}=(z_0,\zeta_0), 
\quad
c_{0,4}=\bigl(x_0,-\tilde{b}_y(y_0,z_0)^T\zeta_0^{\prime\prime}\bigr). 
$$
If $c(s)$ is a curve in $T_{c_0}D_{k,\alpha}$, then there exist 
$u \in T_{x_0}X$, $\mu,\tilde{\mu} \in T_{z_0}\mathcal{G}$, $\nu,\tilde{\nu} \in T_{\zeta_0}(T^\ast_{z_0}\mathcal{G}) $ and $v \in T_{y_0}X$ such that 
$c_1(s)$, $c_2(s)$, $c_3(s)$ and $c_4(s)$ are given by 
\begin{align*}
    c_1(s)
& =
    \Bigl(x_0+su+\mathcal{O}(s^2), -b_x\bigl(x_0+su+\mathcal{O}(s^2),z_0^\prime+s\mu^\prime+\mathcal{O}(s^2)\bigr)\bigl(\zeta_0^{\prime\prime}+s\nu^{\prime\prime}+\mathcal{O}(s^2)\bigr)\Bigr),
\\
    c_2(s)
& =
    \bigl(z_0+s\mu+\mathcal{O}(s^2),\zeta_0+s\nu+\mathcal{O}(s^2)\bigr), 
\\
    c_3(s)
& =
   \bigl(z_0+s\tilde{\mu}+\mathcal{O}(s^2),\zeta_0+s\tilde{\nu}+\mathcal{O}(s^2)\bigr),
\\
   c_4(s)
& =
    \Bigl(y_0+sv+\mathcal{O}(s^2), -\tilde{b}_y\bigl(y_0+sv+\mathcal{O}(s^2),z_0^\prime+s\tilde{\mu}^\prime+\mathcal{O}(s^2)\bigr)\bigl(\zeta_0^{\prime\prime}+s\tilde{\nu}^{\prime\prime}+\mathcal{O}(s^2)\bigr)\Bigr),
\end{align*}
near $s=0$. Taking the differentiation of the above at $s=0$, we deduce that 
\begin{align*}
    \Xi_1
& = 
    \Bigl(
    u, 
    -
    \partial_x\vert_{x=x_0}
   \bigl(b_x(x,z_0^\prime)^T\zeta_0^{\prime\prime}\bigr)u
    -
    \partial_{z^\prime}\vert_{z^\prime=z_0^\prime}\bigl(b_x(x_0,z^\prime)^T\zeta_0^{\prime\prime}\bigr)\mu^\prime
    -
    b_x(x_0,z_0^\prime)^T\nu^{\prime\prime}   
    \Bigr),
\\
    \Xi_2
& = 
   (\mu,\nu),
\\
    \Xi_3
& =
   (\tilde{\mu},\tilde{\nu}),
\\
    \Xi_4
& = 
    \Bigl(
    v, 
    -
    \partial_y\vert_{y=y_0}
   \bigl(\tilde{b}_y(y,z_0^\prime)^T\zeta_0^{\prime\prime}\bigr)v
    -
    \partial_{z^\prime}\vert_{z^\prime=z_0^\prime}\bigl(\tilde{b}_y(y_0,z^\prime)^T\zeta_0^{\prime\prime}\bigr)\tilde{\mu}^\prime
    -
    \tilde{b}_y(y_0,z_0^\prime)^T\tilde{\nu}^{\prime\prime}   
    \Bigr). 
\end{align*}
If we suppose that $\Xi \in T_{c_0}E_{k,\alpha}$ in addition, then we have $\Xi_2= (\mu,\nu)=(\tilde{\mu},\tilde{\nu})=\Xi_3$ 
since $(\Xi_2,\Xi_3)  \in T_{(c_{0,2},c_{0,2})}\bigl(\Delta(T^\ast\mathcal{G}))\bigr)$. 
This completes the proof. 
\end{proof}
%
%
%
%
%
%


\begin{thebibliography}{99}
\bibitem{Chihara1} 
H.~Chihara, 
{\it Microlocal analysis of $d$-plane transform on the Euclidean space}, 
SIAM J. Math.\ Anal., {\bf 54} (2022), pp.6254--6287.

\bibitem{Chihara2} 
H.~Chihara, 
{\it Geodesic X-ray transform and streaking artifacts on simple surfaces or on spaces of constant curvature}, 
submitted, arXiv:2402.06899.

\bibitem{FeizmohammadiIlmavirtaOksanen} 
A.~Feizmohammadi, J.~Ilmavirta, and L.~Oksanen, 
{\it The light ray transform in stationary and static Lorentzian geometries}, 
J. Geom.\ Anal., {\bf 31} (2021), pp.3656--3682. 

\bibitem{GelfandGraevShapiro}
I.~M.~Gelfand, M.~I.~Graev, and Z.~Ja.~Shapiro, 
{\it Differential forms and integral geometry}, 
Funkcional.\ Anal.\ i Priložen., {\bf 3} (1969), pp.24--40 (Russian).

\bibitem{Guillemin}
V.~Guillemin, 
{\it On some results of Gelfand in integral geometry}, 
Proc.\ Symp.\ Pure Math., {\bf 43} (1985), pp.149--155.

\bibitem{GuilleminSternberg}
V.~Guillemin and S.~Sternberg, 
``Geometric Asymptotics'', 
Mathematical Surveys, {\bf 14}, 
American Mathematical Society, 
Providence, RI, 1977. 

\bibitem{Helgason}
S.~Helgason, 
{\it A duality in integral geometry on symmetric spaces}, 
Proc.\ U.S.-Japan Seminar in Differential Geometry (Kyoto, 1965), pp.37--56, 
Nippon Hyoronsha Co., Ltd., Tokyo, 1966. 

\bibitem{Siggi}
S.~Helgason, 
``Integral Geometry and Radon Transforms'', 
Springer-Verlag, New York, NY, 2011.

\bibitem{HolmanUhlmann}
S.~Holman and G.~Uhlmann, 
{\it On the microlocal analysis of the geodesic X-ray transform with conjugate points}, 
J. Diff.\ Geom., {\bf 108} (2018), pp.459--494.

\bibitem{Hoermander3} 
L.~H\"ormander, 
``The Analysis of Linear Partial Differential Operators III'', 
Springer-Verlag, Berlin, 1985. 

\bibitem{Hoermander4} 
L.~H\"ormander, 
``The Analysis of Linear Partial Differential Operators IV'', 
Springer-Verlag, Berlin, 1985. 

\bibitem{KrishnanQuinto}
V.~P.~Krishnan and E.~T.~Quinto, 
{\it Microlocal analysis in tomography}, 
Handbook of mathematical methods in imaging, 
Vol. 1, 2, 3, pp.847--902, Springer-Verlag, New York, 2015.

\bibitem{LassasOksanenStefanovUhlmann} 
M.~Lassas, L.~Oksanen, P.~Stefanov, and G.~Uhlmann, 
{\it The light ray transform on Lorentzian manifolds}, 
Comm.\ Math.\ Phys., {\bf 377} (2020), pp.1349--1379. 

\bibitem{MazzucchelliSaloTzou}
M.~Mazzucchelli, M.~Salo and L.~Tzou, 
{A general support theorem for analytic double fibration transforms}, 
arXiv:2306.05906.

\bibitem{OksanenSaloStefanovUhlmann} 
L.~Oksanen, M.~Salo, P.~Stefanov, and G.~Uhlmann, 
{\it Inverse problems for real principal type operators}, 
Amer.\ J. Math., {\bf 146} (2024), pp.161--240. 

\bibitem{Quinto}
E.~T.~Quinto, 
{\it Singularities of the X-Ray transform and limited data tomography in $\mathbb{R}^2$ and $\mathbb{R}^3$}, 
SIAM J. Math.\ Anal., {\bf 24} (1993), pp.1215--1225. 

\bibitem{Sakai}
T.~Sakai, 
``Riemannian Geometry'', 
Translations of Mathematical Monographs, {\bf 149}, 
American Mathematical Society, 
Providence, RI, 1996.

\bibitem{PlamenGunther1}
P.~Stefanov and G.~Uhlmann, 
{\it Stability estimates for the X-ray transform of tensor fields and boundary rigidity}, 
Duke Math.\ J., {\bf 123} (2004), pp.445--467.

\bibitem{PlamenGunther2}
P.~Stefanov and G.~Uhlmann, 
{\it Boundary rigidity and stability for generic simple metrics}, 
J. Amer.\ Math.\ Soc., {\bf 18} (2005), pp.975--1003.

\bibitem{PlamenGunther3}
P.~Stefanov and G.~Uhlmann, 
{\it The geodesic X-ray transform with fold caustics}, 
Anal.\ PDE., {\bf 5} (2012), pp.219--260.

\bibitem{StefanovUhlmannVasy}
P.~Stefanov, G.~Uhlmann and A.~Vasy, 
{\it Local and global boundary rigidity and the geodesic X-ray transform in the normal gauge}, Ann. of Math., {\bf 194} (2021), pp.1--95.

\bibitem{VasyWang}
A.~Vasy and Y.~Wang, 
{\it On the light ray transform of wave equation solutions}, 
Comm.\ Math.\ Phys., {\bf 384} (2021), pp.503--532.

\bibitem{Warner}
F.~W.~Warner, 
{\it The conjugate locus of a Riemannian manifold}, 
Amer.\ J. Math., {\bf 87} (1965), pp.575--604.

\bibitem{WebberHolmanQuinto}
J.~W.~Webber, S.~Holman and E.~T.~Quinto, 
{\it Ellipsoidal and hyperbolic Radon transforms;
microlocal properties and injectivity}, 
J. Funct.\ Anal., {bf 285} (2023), Article ID 110056, 30 pages. 
\end{thebibliography}
\end{document}